\documentclass{amsart}
\usepackage{graphicx}
\usepackage{amsaddr}
\usepackage{amssymb}
\usepackage{amsmath}
\usepackage{amsthm}
\usepackage{subcaption}
\usepackage{enumerate}
\usepackage{xcolor}
\usepackage{MnSymbol}
\usepackage{bbm}
\usepackage{calligra}
\usepackage{hyperref}
\DeclareMathAlphabet{\mathcalligra}{T1}{calligra}{m}{n}

\newtheorem{thm}{Theorem}[section]
\newtheorem{cor}[thm]{Corollary}

\newtheorem{lem}[thm]{Lemma}
\newtheorem{prop}[thm]{Proposition}
\newtheorem{prob}[thm]{Problem}
\theoremstyle{definition}
\newtheorem{defn}[thm]{Definition}
\newtheorem{rem}[thm]{Remark}
\newtheorem*{defn*}{Definition}
\newtheorem*{rems*}{Remarks}
\newtheorem*{rem*}{Remark}
\newtheorem{ex}[thm]{Example}

\numberwithin{equation}{section}

\renewcommand{\d}{\mathrm{d}}
\newcommand{\rot}{\mathrm{rot}}
\newcommand{\e}{\mathrm{e}}
\renewcommand{\i}{\mathrm{i}}
\makeatletter
\renewcommand{\p@subfigure}{\thefigure} 
\makeatother

\begin{document}

\title[The $k$th Order Preserving Set] {The $k$th Order Preserving Sets and isoperimetric-type inequalities for planar ovals}

\author[M.F. Safarewicz and M. Zwierzy\'nski]{Maksymilian Filip Safarewicz$^{1}$, Michał Zwierzy\'nski$^{2,\pentagram}$}
\address{Warsaw University of Technology\\
Faculty of Mathematics and Information Science\\
ul. Koszykowa 75\\
00-662 Warszawa\\
Poland}
\email{$^{1}$Maksymilian.Safarewicz.stud@pw.edu.pl, ORCID: 0009-0001-5049-1093}
\email[Corresponding author]{$^{2}$Michal.Zwierzynski@pw.edu.pl, ORCID: 0000-0002-9627-1563}

\thanks{$^{\pentagram}$Corresponding author}



\begin{abstract}
In this work, we introduce and investigate a new class of sets, the \textit{$k$th Order Preserving Sets}, arising naturally from the Fourier analysis of support functions associated with hedgehogs. Specifically, we focus on sets whose support functions possess a Fourier series that preserves only terms with positive indices divisible by a fixed $k$.

We explore the geometry of the \textit{$k$th Order Midpoint Set}, defined as the set of centroids of all equiangular $k$-gons circumscribed about a given hedgehog. This set captures essential structural and symmetry-related features of the underlying geometric configuration.

We study the geometric properties of such sets and, in particular, establish an isoperimetric-type inequality relating the perimeter and area of a region bounded by a simple smooth convex closed curve (an oval) $\mathcal{O}$:
\[
L_{\mathcal{O}}^2 - 4\pi A_{\mathcal{O}} \geqslant 4\pi |A_{\mathcal{P}_k}| + 2\pi |A_{\Omega_{\mathcal{O},k}}|,
\]
where $L_{\mathcal{O}}$ denotes the length (perimeter) of $\mathcal{O}$, $A_{\mathcal{O}}$ is the area of the region enclosed by $\mathcal{O}$, $A_{\mathcal{P}_k}$ is the oriented area of the associated $k$th Order Preserving Set $\mathcal{P}_k$, and $A_{\Omega_{\mathcal{O},k}}$ is the oriented area of the associated $k$th Order Midpoint Set $\Omega_{\mathcal{O},k}$. Moreover, we characterize the equality case: the inequality becomes an equality if and only if every equiangular circumscribed $k$-gon around $\mathcal{O}$ is a~regular $k$-gon with its center of mass located at the Steiner point of $\mathcal{O}$.

\end{abstract}

\maketitle

\section*{Statements and Declarations}
Competing Interests: The authors declare that they have no competing interests.

Funding: The authors did not receive support from any organization for the submitted work.

\section{Introduction}

\noindent In recent years, the study of isoperimetric inequalities and related geometric relations has gained renewed momentum, driven by their pivotal role in understanding the behavior of geometric flows, curvature-driven evolutions, and shape optimization problems. These inequalities establish a strong link between classical geometric quantities -- such as perimeter, area, and curvature -- and more recently introduced concepts like the oriented areas of evolutes and Wigner caustics, with applications spanning convex and differential geometry, singularity theory, and mathematical physics -- see \cite{CGR1, G6, G7, GS1, G4, G2, kwonglee, kwong, PanShu, Schmidt, YY,  ZD, Z2, Z3, ZCWMS}, and the literature therein.

In this work, we introduce and investigate a new class of geometric sets, which we call $k$th Order Preserving Sets. These sets arise naturally from the Fourier analysis of support functions associated with a particular class of planar objects known as hedgehogs. Specifically, we focus on support functions whose Fourier series preserves only the terms with positive indices divisible by a fixed integer $k$. This selective frequency retention reveals geometric structures that encode specific symmetry and regularity conditions. Furthermore, we introduce and study the geometry of the \textit{$k$th Order Midpoint Set}, defined as the set of centroids of all equiangular $k$-gons circumscribed about a given hedgehog. By circumscribing a $k$-gon about a hedgehog, we mean that each edge of the polygon lies along a support line of the hedgehog.
One of the central results of our study is a new isoperimetric-type inequality (see Theorem \ref{ThmIsoEq2}), which connects the classical perimeter-area difference of a smooth convex closed curve $\mathcal{O}$ (an oval) with the oriented area of the associated $k$th Order Preserving Set $\mathcal{P}_k$ and $k$th Order Midpoint Set $\Omega_{\mathcal{O},k}$:
\begin{align*}
L_{\mathcal{O}}^2 - 4\pi A_{\mathcal{O}} \geqslant  4\pi |A_{\mathcal{P}_k}| + 2\pi |A_{\Omega_{\mathcal{O},k}}|,
\end{align*}
where $L_{\mathcal{O}}, A_{\mathcal{O}}, A_{\mathcal{P}_k}, A_{\Omega_{\mathcal{O},k}}$ are the length of $\mathcal{O}$, the area bounded by $\mathcal{O}$, the oriented area of the $k$th Order Preserving Set $\mathcal{P}_k$ and that of the $k$th Order Midpoint Set $\Omega_{\mathcal{O},k}$, respectively. We further characterize the case of equality in this inequality: it holds if and only if every equiangular circumscribed $k$-gon about $\mathcal{O}$ is a regular \linebreak $k$-gon whose centroid coincides with the Steiner point of the curve $\mathcal{O}$ -- a fundamental affine-invariant center introduced by Blaschke.

Interestingly, while investigating the equality case in our isoperimetric-type inequality, we found ourselves studying a problem that is, in a certain sense, dual to the classical Square Peg Problem. The Square Peg Problem -- also known as Toeplitz’s conjecture -- asks whether every Jordan curve contains four points that form the vertices of a square. Despite its apparent simplicity, this problem remains open in full generality and has inspired a wealth of research in topological and geometric methods (for a survey of this problem see~\cite{PegSurvey}).

In contrast, our setting takes a different perspective. On the one hand, for convex curves such as ovals, the existence of a circumscribed square is a classical consequence of continuity arguments and can be established easily. On the other hand, we pose a more refined and structured question:
Given that there exist infinitely many equiangular $k$-gons circumscribed about a convex curve or a hedgehog, under what conditions are all of them regular and share a common centroid?

Rather than focusing on the existence of a single configuration, we study the geometry of the entire family of such polygons and the structure of the set of their centroids -- a direction that leads us to the notion of the $k$th Order Midpoint Set. This approach bridges harmonic analysis of support functions with classical polygonal constructions and reveals deep symmetry properties of the underlying curve.

The paper is organized as follows.
 
Section~\ref{sec2} offers a geometrical and analytical foundation for the study.
It introduces the key geometric quantities associated with smooth convex curves and discusses how they can be expressed using the Fourier coefficients of the corresponding support function. This approach enables a seamless transition from classical geometric intuition to harmonic analysis, setting the groundwork for the constructions and inequalities presented in the following sections.

Section~\ref{sec3} develops the core concept of this work, the $k$th Order Preserving Set, defined through the framework of isogonal points (see Definition~\ref{k-tuple pair}).
These sets arise from families of equiangular polygonal configurations, determined by pairs of isogonal points located on a given curve. Subsequently, we analyze the Fourier-theoretic properties of these sets, viewing them as natural consequences of the geometric construction. The~section also explores their structural characteristics, such as singularities, symmetry, and regularity.

Section ~ \ref{sec4} is dedicated to the introduction of the $k$th Order Midpoint Set.
This object is defined as the set of centroids of all equiangular $k$-gons circumscribed about a~given hedgehog. It encodes important information about the symmetry and internal structure of the underlying curve and establishes a natural connection between polygonal configurations and analytic representations via support functions.

Finally, Section ~ \ref{sec5} is divided into two subsections. Subsection \ref{ssec51} focuses on the formulation of the main isoperimetric-type inequalities. We derive a lower bound for the isoperimetric deficit involving the perimeter and area of an oval, expressed in terms of the oriented areas of the corresponding $k$th Order Preserving Set and $k$th Order Midpoint Set.  Additionally, we analyze the equality case. Subsection \ref{ssec52} examines the stability of the aforementioned isoperimetric-type inequalities in certain special cases and investigates how geometric deviations influence a particular instance of the inequality, thereby providing insight into its stability properties.


\section{Geometrical and analytical introduction}\label{sec2}

\noindent In this section, we begin by presenting the essential definitions, formulas, and geometric constructions that will be used throughout the rest of the paper. These foundational elements provide the necessary context for new concepts introduced in later sections.

Let $\mathcal{H}$ denote a \textit{smooth planar curve}, i.e., a $C^\infty$-smooth map from an interval into $\mathbb{R}^2$. The curve $\mathcal{H}$ is said to be \textit{closed} if it is a $C^\infty$ map from the unit circle $S^1$ to $\mathbb{R}^2$. It is called \textit{regular} if its velocity vector never vanishes along the parameter domain. Otherwise, it is called \textit{ singular}. A singular point is a \textit{cusp} (or an \textit{ordinary cusp}) if it is locally diffeomorphic (in the source and in the target) to the map $t\mapsto(t^2,t^3)$ at $t=0$. It is well known (e.g., see Theorem B.9.1 in \cite{UYBook}) that a map $f: S^1 \to \mathbb{R}^2$ has an ordinary cusp at $s_0$ if and only if $f'(s_0)=0$ and $\det\big(f''(s_0),f'''(s_0)\big)\neq 0.$ A regular \textit{simple} (i.e., non-self-intersecting) closed curve is called \textit{convex} if its signed curvature maintains a constant (nonzero) sign throughout. A \textit{hedgehog} is a closed planar curve that can be represented as the boundary of the Minkowski difference of two convex bodies in the plane (see \cite{MMY, MMY2}). This construction allows hedgehogs to generalize support functions beyond the convex setting. An \textit{oval}, in contrast, is a simple, smooth, regular, and closed convex curve. Ovals play a central role in classical convex geometry and serve as natural domains for the Fourier analysis of support functions.
For every direction $u(s) = (\cos (s), \sin (s)) \in S^1$, the \textit{support function} $h: [0,2\pi] \to \mathbb{R}$ of a~hedgehog $\mathcal{H}$ is defined as a smooth function in the following way:
\begin{align*}
h(s):= \sup_{x \in \mathcal{H}} \langle x, u \rangle,\
\end{align*}
where $\langle \cdot, \cdot \rangle$ denotes the standard Euclidean inner product.
If $\mathcal{H}$ is represented as the Minkowski difference $\mathcal{H} = \mathcal{K}_1 - \mathcal{K}_2$, where $\mathcal{K}_1, \mathcal{K}_2$ are ovals, then the support function satisfies:
$$
h(s) = h_{\mathcal{K}_1}(s) - h_{\mathcal{K}_2}(s),
$$
where $h_{\mathcal{K}_1}$ and $h_{\mathcal{K}_2}$ are the classical support functions of $\mathcal{K}_1$ and $\mathcal{K}_2$, respectively.

Given a $2\pi$-periodic smooth function $h(s)$, one can obtain the parameterization of a hedgehog $\mathcal{H}$ as follows:
\begin{align}
\label{eqHedghehogForm}\mathcal{H}(s) &= h(s)\, u(s) + h'(s)\, u'(s)\\ 
\nonumber&=\Big( h(s) \cos(s) - h'(s) \sin(s),  h(s) \sin(s) + h'(s) \cos(s) \Big),
\end{align}
where $u(s) = (\cos (s), \sin (s))$ and $u'(s) = (-\sin (s), \cos (s))$ is the vector $u(s)$ rotated by $\frac{\pi}{2}$ counterclockwise.

A \textit{support line} of a hedgehog $\mathcal{H}$ at the point $\mathcal{H}(s)$ is defined as the line passing through $\mathcal{H}(s)$ and orthogonal to $u(s)$. Geometrically, the support function $h(s)$ can be interpreted as the oriented distance from the origin to this support line.

\begin{rem}
If $h(s)+h''(s)$ has a constant sign for all $s$, then the resulting curve $\mathcal{H}$ is convex. Otherwise, $\mathcal{H}$ is a hedgehog that exhibits singularities or self-intersections.
\end{rem}

The \textit{Steiner point} is a distinguished point associated with a hedgehog. It is an~affine-invariant center that plays a fundamental role in convex geometry, especially in connection with support functions and affine symmetrization processes.

\begin{defn}
Let $\mathcal{H}$ be a hedgehog. The \textit{Steiner point} of $\mathcal{H}$, denoted by $S(\mathcal{H})$, is the point
\[
S(\mathcal{H}) = \frac{1}{\pi} \int_0^{2\pi} h(s) \, u(s) \, \d s,
\]
where $u(s) = (\cos (s), \sin (s))$. This definition coincides with the classical notion of the Steiner point for convex bodies (see, e.g., \cite{SchneiderBook}), and extends naturally to hedgehogs via their support functions.
\end{defn}

Let $\mathcal{H}$ be a hedgehog. We denote by $\rot(\mathcal{H}, \alpha)$ the hedgehog obtained by rotating $\mathcal{H}$ counterclockwise about the origin by an angle $\alpha$. In particular, we define $\mathcal{H}^{\perp}:=\rot\left(\mathcal{H},\frac{\pi}{2}\right)$. If $h(s)$ is a support function of $\mathcal{H}$, then $h(s-\alpha)$ is a support function of $\rot\left(\mathcal{H},\alpha\right)$, albeit the parameter $s$ is translated. In particular:
\begin{align}
    \label{param}\mathcal{H}^\perp(s) = -\Big( h(s) \sin(s) + h'(s) \cos(s),  -h(s) \cos(s) + h'(s) \sin(s) \Big).
\end{align}

\begin{defn}
    Let $\mathcal{H}$ be a hedgehog. Then the \textit{width} of $\mathcal{H}$ in the direction $u(s)$, with $s\in[0,2\pi]$, is
    $$w(s) = h(s)+h(s+\pi),$$
    i.e., it is the oriented distance between the two support lines with normals $\pm u(s)$.
    Furthermore, the \textit{average width} of hedgehog $\mathcal{H}$ is
    $$\overline{w} = \frac{1}{\pi}\int_0^{\pi}w(s)\,\d s=\frac{1}{\pi}\int_0^{2\pi}h(s)\,\d s.$$
\end{defn}

The \textit{Steiner disk} $D_{\mathcal{H}}$ associated with a hedgehog $\mathcal{H}$ is the unique Euclidean disk centered at the Steiner point $S(\mathcal{H})$ and having radius $\frac{1}{2}\left|\overline{w}\right|$. In the degenerate case $\overline{w}=0$, the Steiner disk reduces to the single point $S(\mathcal{H})$.
Next, we introduce a~natural generalization of the concept of area of a region to the oriented area of a~closed curve.
\begin{defn}
Let $\mathcal{O}$ be an oriented smooth closed curve. The \textit{oriented area} of $\mathcal{O}$ is defined by
\begin{align*}
A_\mathcal{O}:=\frac{1}{2}\int_\mathcal{O}-y\,\d x+x\,\d y=\iint_{\mathbb{R}^2}w_\mathcal{O}(x,y)\,\d x\,\d y,
\end{align*}
where $w_\mathcal{O}(x,y)$ is the winding number of $\mathcal{O}$ about the point $(x,y)\in\mathbb{R}^2$.
\end{defn}

Now, since the support function $h$ is smooth and $2\pi$-periodic, we express it via its Fourier series expansion as follows:
\begin{align*}
    h(s)=a_0+\sum_{n=1}^{\infty}\left(a_n\cos \left(ns\right)+b_n\sin \left(ns\right)\right).
\end{align*}
One can check that the area of the oval $\mathcal{O}$ (\textit{Cauchy's formula}), the length of $\mathcal{O}$ (\textit{Blaschke's formula}), the width and average width of $\mathcal{O}$, and the Steiner point of $\mathcal{O}$ are described by the following formulae in terms of the coefficients of the Fourier series of its Minkowski support function $h$ (see Section 4.2 in \cite{G4}):
\begin{align}
\label{OAreaFourier}A_\mathcal{O} &=\frac{1}{2}\int_0^{2\pi}\Big(h^2(s)-h'^2(s)\Big)\,\d s=\pi a_0^2-\frac{\pi}{2}\sum_{\substack{n=2}}^{\infty}(n^2-1)(a_n^2+b_n^2),\\
\label{LengthFourier}L_\mathcal{O} &=\int_0^{2\pi}h(s)\,\d s =2\pi a_0,\\
\label{width} \overline{w} & = 2a_0, \\
\label{eq:SteinerPoint}S(\mathcal{O})&=(a_1,b_1).
\end{align}
Note that formulae \eqref{width} and \eqref{eq:SteinerPoint} are also valid
for a hedgehog $\mathcal{H}$ instead of an~oval $\mathcal{O}$.
We call the quantity 
\begin{equation}\label{isdef}
    L^2_\mathcal{O} - 4\pi A_\mathcal{O} = 2\pi^2\sum_{n=2}^{\infty}(n^2-1)(a_n^2 + b_n^2) 
\end{equation}
the \textit{isoperimetric deficit}.
Furthermore, in the case of a hedgehog $\mathcal{H}$, one gets an~analogous formula for its oriented area:
\begin{align}
\label{HAreaFourier}A_\mathcal{H} &=\frac{1}{2}\int_0^{2\pi}\Big(h^2(s)-h'^2(s)\Big)\,\d s=\pi a_0^2-\frac{\pi}{2}\sum_{\substack{n=2}}^{\infty}(n^2-1)(a_n^2+b_n^2).
\end{align}

In recent years, the study of hedgehogs has gained considerable attention within the
fields of convex geometry, differential geometry, and singularity theory (see, e.g., 
\cite{kwong, MMYLorentzian, MMYDR, Dominika, Mozgawa, Rochera2, Rochera1, SC1, Taba, Z2, ZCWMS} for both
Euclidean and non-Euclidean treatments). This renewed interest has led not only to new isoperimetric-type inequalities but also to systematic investigations of hedgehogs themselves, including their intrinsic quantities and the structure of their singularities.

Furthermore, hedgehog-related constructions have appeared in the formulation and refinement of generalized versions of the Gauss--Bonnet theorem, where curvature and topological invariants are analyzed through the lens of support functions and their harmonic structure. These developments underscore the growing role of hedgehogs as a unifying concept bridging classical geometry, Fourier analysis, and global geometric invariants (see \cite{DZgb, Z4}).


\section{The $k$th Order Preserving Set }\label{sec3}

\noindent In this section, we introduce and study a new class of geometric objects that emerge naturally from the Fourier analysis of support functions -- the $k$th Order Preserving Sets. These sets arise when one imposes a specific harmonic constraint on the support function of a hedgehog: namely, that only the positive-frequency Fourier components whose indices are divisible by a fixed integer $k$ are retained.

This seemingly simple restriction leads to rich geometric behavior and reveals surprising structural phenomena. The resulting sets generalize classical notions of symmetry, regularity, and curvature, while also offering a new perspective on the relationship between analytic representations (via the Fourier series) and geometric shape.

From a broader perspective, this construction can be seen as a type of Fourier projection onto a $k$-harmonic subspace, which preserves only selected symmetries and periodicities of the original curve. This idea resonates with classical problems in geometric tomography, spectral geometry, and affine differential geometry.

\begin{defn}\label{k-tuple pair}
Let $\mathcal{H}$ be a hedgehog and let $k>2$ be an integer. We call points
$p_1,\ldots,p_k\in\mathcal{H}$ an \textit{isogonal family of points} if, for each
$i=1,\ldots,k$ (with the convention $p_{k+1}=p_1$), the support lines to
$\mathcal{H}$ at $p_i$ and $p_{i+1}$ meet at the same angle, that is, $\frac{2\pi}{k}$.
Equivalently, the unit normals (compatible with the orientation of $\mathcal{H}$) at $p_1,\ldots,p_k$ are spaced by angles~$\frac{2\pi}{k}$.

Let $\mathcal{H}^{\perp}$ be the image of $\mathcal{H}$ under a counterclockwise
rotation by $\frac{\pi}{2}$ about the origin. Consider the parameterization
of $\mathcal{H}$ given by equation~\eqref{eqHedghehogForm} and the induced
parameterization of $\mathcal{H}^{\perp}$ as in~\eqref{param}. For $s\in[0,2\pi]$,
write $p=\mathcal{H}(s)$ and define $p^{\perp}=\mathcal{H}^{\perp}(s)$. We denote by
$p_1^{\perp},\ldots,p_k^{\perp}\in\mathcal{H}^{\perp}$ the images of an isogonal family
$p_1,\ldots,p_k\in\mathcal{H}$; in particular, for every $i$, the support line to
$\mathcal{H}^{\perp}$ at $p_i^{\perp}$ is perpendicular to the support line to
$\mathcal{H}$ at $p_i$.
\end{defn}
 
\begin{ex}
We consider an oval $\mathcal{O}$ with a support function $h(s) = 20 + \sin(2s) + \cos(3s) $ and $k=3$ (see Figure~\ref{fig: fig1}).

\begin{figure}[h]
\centering
\includegraphics[scale=0.22]{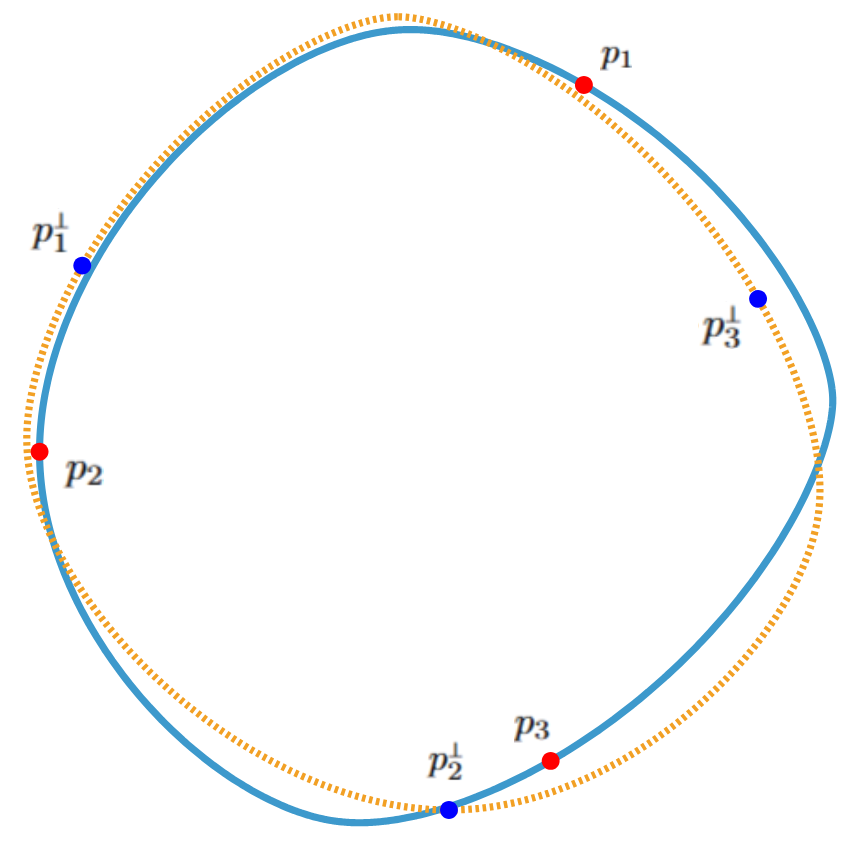}
\caption{An oval $\mathcal{O}$ (solid) with $p_1, p_2, p_3$ an isogonal family of points (red) and $\mathcal{O}^{\perp}$ (dashed) with $p_1^{\perp}, p_2^{\perp}, p_3^{\perp}$ (blue)} 
\label{fig: fig1}
\end{figure} 
\end{ex}
Note that, as specified in Definition~\ref{k-tuple pair}, the set of isogonal families 
of points on $\mathcal{H}$ constitutes a one-parameter family, naturally induced by the 
parameterization of $\mathcal{H}$. In particular, the point $p_1=\mathcal{H}(0)$ gives 
rise to an entire isogonal family. Proceeding along the parameter, we reach the isogonal 
family associated with each $p_1=\mathcal{H}\left(\tfrac{2\pi}{k}\right)$. Based on this 
observation, we introduce the titular set, built from this one-parameter family of 
isogonal points.

\begin{defn}\label{preserving set}
Consider a hedgehog $\mathcal{H}$. Then for a given integer $k>2$, the \textit{$k$th Order Preserving Set} is defined by
\begin{align}
\label{eq:preservdefn}\mathcal{P}_k:=\Bigg\{\frac{1}{k}\sum_{j=1}^{k}\Bigg(\cos\left(\frac{2\pi j}{k}\right)\cdot p_j-\sin\left(\frac{2\pi j}{k}\right)\cdot p_j^{\perp}\Bigg) -\frac{1}{2}\overline{w}\cdot \mathbbm{n}(p_1)\colon\\ 
\nonumber \{p_j\}_{j=1}^{j=k} \text{ is an isogonal family of points on $\mathcal{H}$ }\Bigg\},
\end{align}
where $\overline{w}$ denotes the average width of $\mathcal{H}$, and $\mathbbm{n}(p)$ is 
a continuous unit normal vector field to $\mathcal{H}$ in the direction given by the coorientation of the support line to $\mathcal{H}$ at the point $p$.
\end{defn}

Assume that the coordinates of the point $p_j$ are $(x_j,y_j)$. Then, clearly, 
$p_j^\perp = (-y_j,x_j)$. Hence, we see that
{\footnotesize$$\cos\left(\frac{2\pi j}{k}\right)\cdot p_j-\sin\left(\frac{2\pi j}{k}\right)\cdot p_j^{\perp} = \left(\cos\left(\frac{2\pi j}{k}\right)x_j + \sin\left(\frac{2\pi j}{k}\right)y_j,\cos\left(\frac{2\pi j}{k}\right)y_j - \sin\left(\frac{2\pi j}{k}\right)x_j\right).$$}
Now, recalling the standard rotation formula
\[
\rot\big((x,y),\alpha\big) = \big(\cos(\alpha)\, x - \sin(\alpha)\, y,\,
\cos(\alpha)\, y + \sin(\alpha)\, x\big),
\]
we obtain the following equivalent formulation of \eqref{eq:preservdefn}:
\begin{align*} \mathcal{P}_k=&\Bigg\{\frac{1}{k}\sum_{j=1}^{k}\rot\left(p_j,-\frac{2\pi j}{k}\right)-\frac{1}{2}\overline{w}\cdot \mathbbm{n}(p_1)\colon\\ &\{p_j\}_{j=1}^{j=k} \text{ is an isogonal family of points of $\mathcal{H}$ }\Bigg\}. \end{align*}

\begin{rem}\label{remCWMS}
In this paper, we focus on the case $k>2$. For $k=2$, the corresponding “isogonal family” 
can be interpreted as a pair of distinct points $p_1,p_2\in\mathcal{H}$ whose support 
lines are parallel (a so-called \textit{parallel pair}). In this situation, the associated $2$nd 
Order Preserving Set $\mathcal{P}_2$ coincides with a homothetic copy of the Constant 
Width Measure Set with ratio $\frac{1}{2}$, which has been defined and studied in \cite{ZCWMS}, and then generalized in \cite{SC1, Z3}.
\end{rem}

\begin{ex}
    Consider an oval $\mathcal{O}$ with a support function $h(s) = 30 + \sin(2s) + \cos(3s)+ \cos(4s) $ and its $3$rd and $4$th Order Preserving Sets (see Figure~\ref{fig: fig2}). Note that the number of cusps in these examples is divisible by $k$ as it should be (see~Proposition~ \ref{singularitiesofT}).

\begin{figure}[h]
\centering
\includegraphics[scale=0.27]{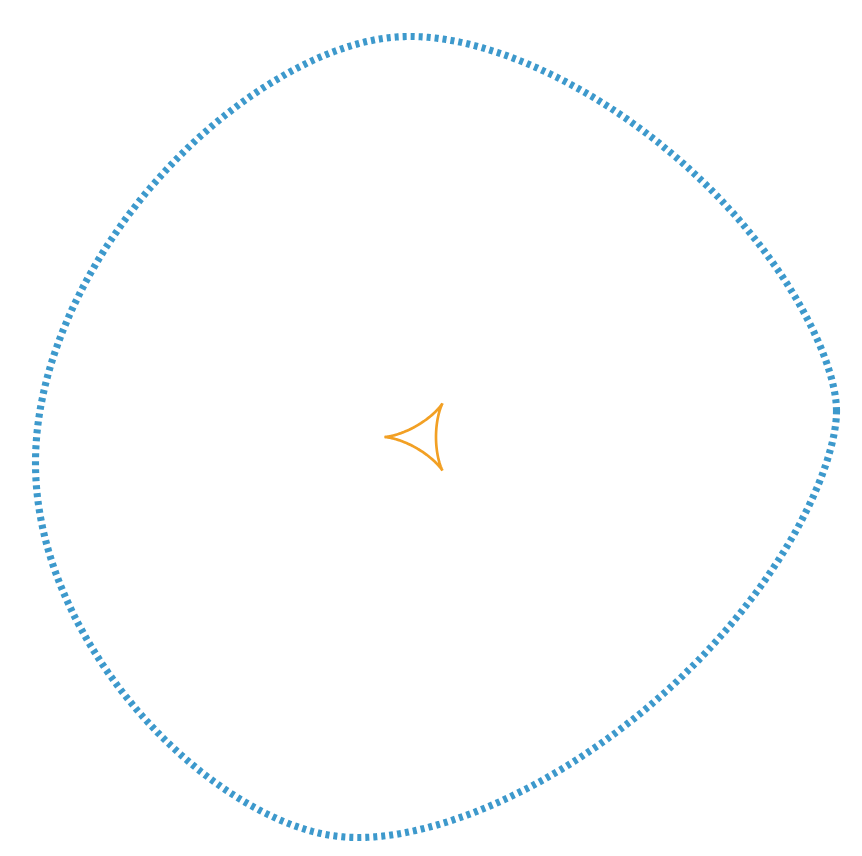}
\includegraphics[scale=0.27]{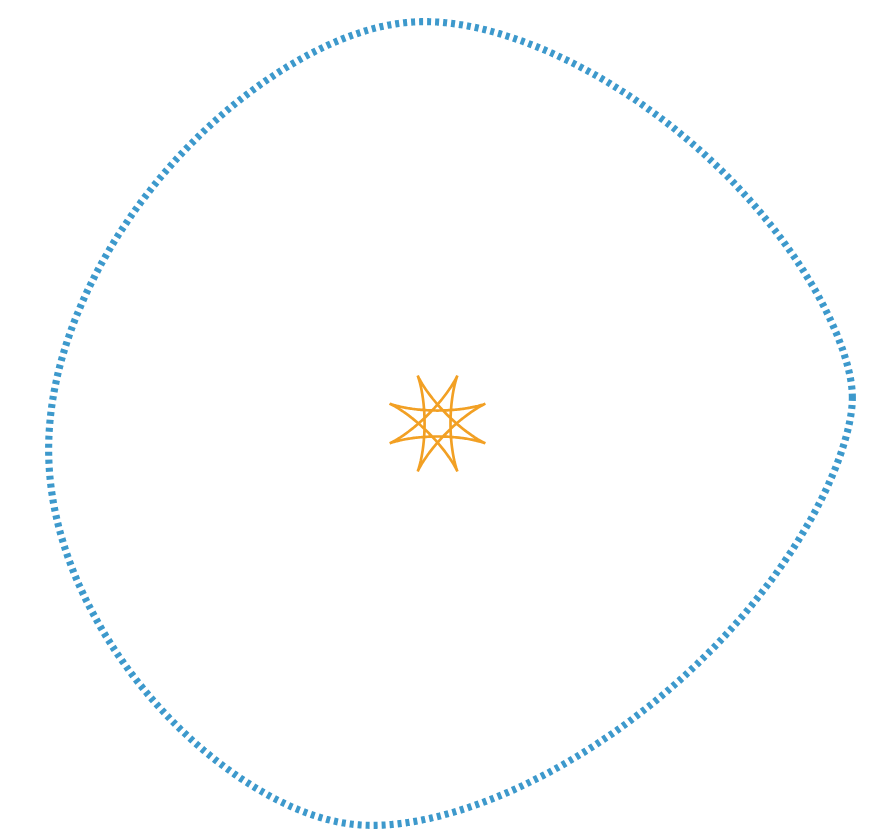}
\caption{On the left, oval $\mathcal{O}$ (blue, dashed) with its $3$rd Order Preserving Set (orange); 
on the right, $\mathcal{O}$ with its $4$th Order Preserving Set (orange)}
\label{fig: fig2}
\end{figure} 
\end{ex}

\begin{ex}
    Consider a singular hedgehog $\mathcal{H}$ with a support function \linebreak $h(s) = 10 + \sin(2s) + \cos(3s)+ \cos(6s) $ and its $3$rd Order Preserving Set (see Figure~\ref{fig: fig3}).

    \begin{figure}[h]
\centering
\includegraphics[scale=0.29]{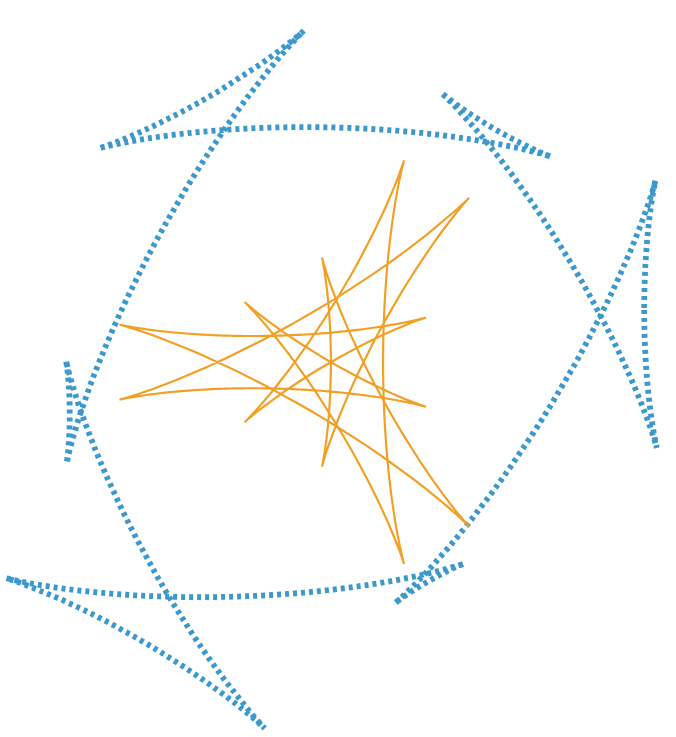}
\caption{A hedgehog $\mathcal{H}$ (blue, dashed) and its $3$rd Order Preserving Set (orange)}
\label{fig: fig3}
\end{figure} 
\end{ex}

\begin{prop}\label{Preservingsuppth}
Let $\mathcal{H}$ be a hedgehog with the support function $h: [0,2\pi] \to \mathbb{R}$, where $h(s)=a_0+\sum_{n=1}^{\infty}a_n\cos (ns)+b_n\sin (ns)$ is its Fourier series representation, and let the average width of $\mathcal{H}$ be $\overline{w}$. Then, its $k$th Order Preserving Set is a~hedgehog with a support function given by the formula:
$$ h_k(s) = \frac{1}{k}\sum_{j=0}^{k-1}h\left(s+\frac{2\pi j}{k}\right) - \frac{1}{2}\overline{w},$$
and its Fourier series representation is as follows:
\begin{align}
    \label{eqFourierkOPS}h_k(s) = \sum_{k \mid n,\ n > 1}^\infty a_n\cos(ns) + b_n\sin(ns).
\end{align}
\end{prop}
\begin{proof}
 Note that the points $\mathcal{H}(s),\ldots,\mathcal{H}\left(s+\frac{2\pi(k-1)}{k}\right)$, for a given $s\in[0,2\pi]$, constitute an isogonal family of $\mathcal{H}$. Then, the parameterization of the $k$th Order Preserving Set is:
    $$\mathcal{P}_k(s) =\frac{1}{k}\sum_{j=1}^{k}\Bigg(\cos\left(\frac{2\pi j}{k}\right)\cdot \mathcal{H}\left(s+\frac{2\pi j}{k}\right)-\sin\left(\frac{2\pi j}{k}\right)\cdot \mathcal{H}^\perp\left(s+\frac{2\pi j}{k}\right)\Bigg) -\frac{1}{2}\overline{w}\cdot u(s).$$\\Naturally, the point $\mathcal{H}^\perp(s+\frac{2\pi j}{k})$ is understood as in Definition \ref{k-tuple pair}. Using the parameterization \eqref{param} of $\mathcal{H}^\perp$ we obtain the following equalities:
    {\footnotesize
    \begin{align*}    
    \mathcal{P}_k&(s) = \frac{1}{k}\sum_{j=1}^{k}\Bigg(\cos\Big(\frac{2\pi j}{k}\Big)\cdot \Bigg( h\Big(s+\frac{2\pi j}{k}\Big)\cos\Big(s+\frac{2\pi j}{k}\Big) - h'\Big(s+\frac{2\pi j}{k}\Big) \sin\Big(s+\frac{2\pi j}{k}\Big),\\&  h\Big(s+\frac{2\pi j}{k}\Big) \sin\Big(s+\frac{2\pi j}{k}\Big) + h'\Big(s+\frac{2\pi j}{k}\Big) \cos\Big(s+\frac{2\pi j}{k}\Big) \Bigg)+\\&+\frac{1}{k}\sum_{j=1}^{k}\Bigg(\sin\Big(\frac{2\pi j}{k}\Big)\cdot \Bigg( h\Big(s+\frac{2\pi j}{k}\Big) \sin\Big(s+\frac{2\pi j}{k}\Big) + h'\Big(s+\frac{2\pi j}{k}\Big) \cos\Big(s+\frac{2\pi j}{k}\Big),\\&   -h\Big(s+\frac{2\pi j}{k}\Big) \cos\Big(s+\frac{2\pi j}{k}\Big) + h'\Big(s+\frac{2\pi j}{k}\Big) \sin\Big(s+\frac{2\pi j}{k}\Big) \Bigg)\Bigg)-\frac{1}{2}\overline{w}\cdot \big(\cos(s),\sin(s)\big)\\&=\frac{1}{k}\sum_{j=1}^{k}\Bigg( \Bigg( h\Big(s+\frac{2\pi j}{k}\Big)\cos(s) - h'\Big(s+\frac{2\pi j}{k}\Big) \sin(s),\\&  h\Big(s+\frac{2\pi j}{k}\Big) \sin(s) + h'\Big(s+\frac{2\pi j}{k}\Big) \cos(s) \Bigg)-\frac{1}{2}\overline{w}\cdot \big(\cos(s),\sin(s)\big)\\&= \Big(\cos(s),\sin(s)\Big)\cdot \Bigg(\frac{1}{k}\sum_{j=1}^{k} h\Big(s+\frac{2\pi j}{k}\Big)-\frac{1}{2}\overline{w}\Bigg) + \Big(-\sin(s),\cos(s)\Big)\cdot \Bigg(\frac{1}{k}\sum_{j=1}^{k} h'\Big(s+\frac{2\pi j}{k}\Big)\Bigg).
    \end{align*}}

    Now, setting $h_k(s) = \frac{1}{k}\sum_{j=0}^{k-1}h\left(s+\frac{2\pi j}{k}\right) - \frac{1}{2}\overline{w}$, we obtain $P_k(s)$ in the form of equation \eqref{eqHedghehogForm}. Therefore, we end the first part of the proof. To derive equation \eqref{eqFourierkOPS}, one can simplify the aforementioned form of $h_k$ by leveraging the properties of Fourier and geometric series. Let us consider the Fourier expansion of $h$:
    $$h(s) = \sum_{n=-\infty}^\infty c_n{\e}^{\i ns}.$$
    Hence:
     $$h\left(s+\frac{2\pi j}{k}\right) = \sum_{n=-\infty}^\infty c_n{\e}^{\i n\left(s+\frac{2\pi j}{k}\right)}.$$
     This series is convergent. Thus, the following is well-defined:
     $$\sum_{j=0}^{k-1}h\left(s+\frac{2\pi j}{k}\right) = \sum_{j=0}^{k-1}\sum_{n=-\infty}^\infty c_n{\e}^{\i n\left(s+\frac{2\pi j}{k}\right)}=\sum_{n=-\infty}^\infty \sum_{j=0}^{k-1} c_n{\e}^{\i n\left(s+\frac{2\pi j}{k}\right)}.$$ Notice that $\sum_{n=-\infty}^\infty \sum_{j=0}^{k-1} c_n{\e}^{\i n\left(s+\frac{2\pi j}{k}\right)} = \sum_{n=-\infty}^\infty c_n{\e}^{\i ns}\sum_{j=0}^{k-1} {\e}^{\i\frac{2\pi jn}{k}}$. We shall now simplify the geometric series:
     $$\sum_{j=0}^{k-1} {\e}^{\i\frac{2\pi jn}{k}} = \sum_{j=0}^{k-1} \left({\e}^{\i\frac{2\pi n}{k}}\right)^j = \begin{cases} 
k & \text{if } k\mid n, \\ 
0 & \text{if } k\nmid n. 
\end{cases}$$ Therefore, we arrive at:
$$\sum_{j=0}^{k-1}h\left(s+\frac{2\pi j}{k}\right) = k\sum_{k\mid n} c_n{\e}^{\i ns}.$$ Translating the result into trigonometric series one gets that:
\begin{align*}
\frac{1}{k}\sum_{j=0}^{k-1}h\left(s+\frac{2\pi j}{k}\right) - \frac{1}{2}\overline{w} &= \sum_{k \mid n} a_n\cos(ns) + b_n\sin(ns) - a_0\\ &= \sum_{k \mid n,\ n > 1}a_n\cos(ns) + b_n\sin(ns),
\end{align*}
which concludes the proof.
    \end{proof}
    
By Proposition \ref{Preservingsuppth}, we may parameterize the $k$th Order Preserving Set as follows:
\begin{align}
    \label{eq:pkparameterization}\mathcal{P}_k(s)= \Big( h_k(s) \cos(s) - h_k'(s) \sin(s),  h_k(s) \sin(s) + h_k'(s) \cos(s) \Big),\ s\in[0,2\pi].
\end{align}

\begin{rem}\label{remSteiBall}
    It is worth noting that as $k$ tends to infinity, the Fourier coefficients of $h_k$ vanish. Hence, the limit $\mathcal{P}_\infty$ is the origin.
\end{rem}

\begin{rem}\label{Othersupp}
Consider the Fourier expansion of the support function $h(s) = a_0 + \sum_{n = 1}^\infty a_n\cos(ns) + b_n\sin(ns), \; s\in[0,2\pi]$. By Proposition \ref{Preservingsuppth}, the support function $h_k$ of the $k$th Order Preserving Set only preserves the Fourier indices divisible by $k$, with the exception of $0$. Thus, by equation \eqref{eqFourierkOPS} one gets that:
$$h_k'(s) = \sum_{k \mid n,\; n > 1} n\Big(b_n\cos(ns)-a_n\sin(ns)\Big)$$ and
$$h_k''(s) =-\sum_{k \mid n,\; n > 1}n^2\Big(a_n\cos(ns) + b_n\sin(ns)\Big).$$
\end{rem}

\begin{rem}
Let $\mathcal{H}$ be a hedgehog.
 By analysis of Fourier series, it is evident that $\mathcal{P}_k$ is a point if and only if the support function of $\mathcal{H}$ is of the form: $h(s) = a_0 + \sum_{ k\nmid n }a_n \cos(ns) +b_n \sin(ns)$.
By Definition \ref{preserving set}, the set $\mathcal{P}_k$ is a point if and only if $\frac{1}{k}\sum_{j=1}^{k}\left(\cos\left(\frac{2\pi j}{k}\right)\cdot p_j -\sin\left(\frac{2\pi j}{k}\right)\cdot p_j^\perp\right)=\frac{1}{2}\overline{w}\cdot \mathbbm{n}(p_1)$ for all isogonal families of points $p_1,\ldots, p_k$, which means that the length of the vector $$\sum_{j=1}^{k}\left(\cos\left(\frac{2\pi j}{k}\right)\cdot p_j -\sin\left(\frac{2\pi j}{k}\right)\cdot p_j^\perp\right)$$ is always constant and equal to $\frac{k\overline{w}}{2}$.
\end{rem}

\begin{prop}\label{PropTangentLins}
Consider a hedgehog $\mathcal{H}$, its $k$th Order Preserving Set $\mathcal{P}_k$, and their respective parameterizations. Then, the support line to $\mathcal{H}$ at $\mathcal{H}(s)$ (for some $s\in[0,2\pi]$) is parallel to the support line to $\mathcal{P}_k$ at the point $\mathcal{P}_k(s)$.
 \end{prop}
 \begin{proof}
Differentiating the parameterizations at some \(s \in [0, 2\pi]\) immediately yields the desired result:
\begin{align*}
\mathcal{H}'(s) &=\Big(h(s)+h''(s)\Big)\cdot\Big(-\sin(s),\cos(s)\Big),\\
\mathcal{P}'_k(s) &=\Big(h_k(s)+h''_k(s)\Big)\cdot\Big(-\sin(s),\cos(s)\Big).
\end{align*}  
 \end{proof}

 \begin{figure}[h]
\centering
\includegraphics[scale=0.4]{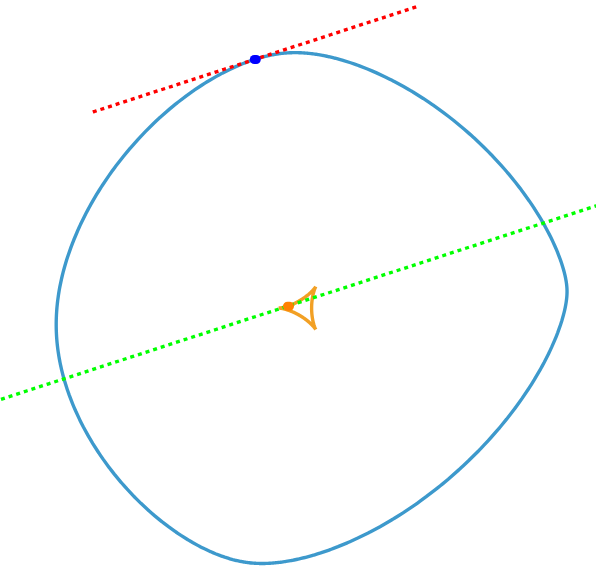}
\caption{Consider an oval $\mathcal{O}$ with a support function $h(s) = 30 + \sin(2s) + \cos(3s)+ \cos(4s) $. The support line to $\mathcal{O}$ at $\mathcal{O}(s)$ is parallel to the support line to $\mathcal{P}_3$ at $\mathcal{P}_3\left(s\right)$  }
\label{fig: fig4}
\end{figure} 
 
We illustrate Proposition \ref{PropTangentLins} in Figure~\ref {fig: fig4}.

\noindent Let us now state a simple observation considering the behavior of periodic functions:
\begin{prop}\label{kperiodic}
    Consider a $2\pi$-periodic smooth function $f$ and a natural number $k > 2$. The function $f$ is $\frac{2\pi}{k}$-periodic if and only if 
    \begin{align}
    \label{eq:2pikperiodic}\forall{x\in[0,2\pi]}\ f(x) = a_0 + \sum_{k\mid n, \;n>1}a_n\cos(nx)+b_n\sin(nx).\end{align}
\end{prop}
\begin{proof}
Let us start by observing that $f$ being of the form \eqref{eq:2pikperiodic}, implies \linebreak $\frac{2\pi}{k}$-periodicity because:
\begin{align*}
    f\left(x+\frac{2\pi}{k}\right) &= a_0 +\sum_{k\mid n, \;n>1}a_n\cos\left(nx+n\frac{2\pi}{k}\right)+b_n\sin\left(nx+ n\frac{2\pi}{k}\right) \\&= a_0+\sum_{k\mid n, \;n>1}a_n\cos(nx)+b_n\sin(nx) = f(x).
\end{align*}
    Since $f$ is a $2\pi$-periodic, smooth function, we may consider its Fourier series:
    $$ f(x) = a_0 + \sum_{n=-\infty}^\infty c_n{\e}^{\i nx}$$
    Let $f$ be a $\frac{2\pi}{k}$-periodic function, hence for all $x \in [0,2\pi]$ we have:
    $$f(x) = f\left(x + \frac{2\pi}{k}\right)= \cdots=f\left(x + \frac{2\pi(k-1)}{k}\right) .$$
    Hence, we see the following:
    $$kf(x) = f\left(x + \frac{2\pi}{k}\right)+ \cdots+f\left(x + \frac{2\pi k}{k}\right).$$
    Consider the Fourier expansion of $f$, one gets the following:
    \begin{align*}
        ka_0 + k\sum_{n=-\infty}^\infty c_n{\e}^{\i nx}&=a_0 + \sum_{n=-\infty}^\infty c_n{\e}^{\i n(x+\frac{2\pi}{k})} +\cdots +a_0 + \sum_{n=-\infty}^\infty c_n{\e}^{\i n\left(x+\frac{2\pi k}{k}\right)}\\
        &= ka_0 + \sum_{n = -\infty}^\infty c_n {\e}^{\i nx}\sum_{j=1}^{k}{\e}^{\i n\frac{2\pi j}{k}}.
    \end{align*}
    Leveraging the properties of geometric series, we obtain:
    \begin{align*}
        k\sum_{n=-\infty}^\infty c_n{\e}^{\i nx}
        &= k\sum_{k\mid n} c_n {\e}^{\i nx} 
        .
    \end{align*}
    This means that the Fourier series preserves only indices that are divisible by $k$. Rewriting this exponential Fourier series into a trigonometric one, immediately yields the desired result.
\end{proof}

\begin{prop}\label{singularitiesofT}
  Let $\mathcal{H}$ be a hedgehog and let $k>2$. Then, the $k$th Order Preserving Set of $\mathcal{H}$, $\mathcal{P}_k$, is singular. Moreover, if $\mathcal{P}_k$ has finitely many singular points, the number of them is divisible by $k$.
\end{prop}
\begin{proof}
    Consider the parameterization \eqref{eq:pkparameterization} of $\mathcal{P}_k$. Differentiating this parameterization, one gets that
    $$\mathcal{P}_k'(s)= \Big(h_k(s) + h_k''(s)\Big)\cdot\Big(-\sin(s),\cos(s)\Big).$$

    Let $\rho(s):=h_k(s)+h_k''(s)$. Note that a zero of $\rho$ corresponds to a singular point of $\mathcal{P}_k$. Since the Fourier series of $h_k$ is \eqref{eqFourierkOPS}, by Proposition \ref{kperiodic}, it is easy to see that the function $h_k$ is $\frac{2\pi}{k}$-periodic. The same holds for $h''_k$, and then for $\rho$. Moreover,
    $$\int_0^{2\pi} \rho(s)\,\d s
  =\int_0^{2\pi} h_k(s)\,\d s
   +\underbrace{\int_0^{2\pi} h_k''(s)\,\d s}_{=h_k'(2\pi)-h_k'(0)=0}
  =0,$$
    since the constant term of the series of $h_k$ vanishes. The mean value of $\rho$ over the interval $[0,2\pi]$ is zero. Hence, $\rho$ has at least one zero, which is equivalent to the fact that $\mathcal{P}_k$ is singular.

    Now assume that $\mathcal P_k$ has finitely many singular points. Equivalently, $\rho$ has only finitely many zeros in $[0,2\pi)$. Since $\rho$ is $\frac{2\pi}{k}$-periodic, whenever $s_0$ is zero, so are 
    $$
    s_0+\tfrac{2\pi}{k},\ s_0+\tfrac{2\cdot 2\pi}{k},\ \ldots,\
    s_0+\tfrac{(k-1)\,2\pi}{k}\pmod{2\pi}.
    $$
    These $k$ values are distinct modulo $2\pi$, so zeros occur in $k$-tuples. Hence, the total number of zeros of $\rho$ in $[0,2\pi)$ -- and therefore the number of singular points of $\mathcal P_k$ -- is divisible by $k$, which completes the proof.
    \end{proof}

    \begin{thm}\label{Thmcurvature}
Let $\mathcal{H}$ be a hedgehog with an average width equal to $\overline{w}$. Let $\rho_{\mathcal{H}}(p)$ denote the radius of curvature of $\mathcal{H}$ at $p$. Let $p_1,\ldots, p_k$ be an isogonal family of points in $\mathcal{H}$ and let $q=\frac{1}{k}\sum_{j=1}^{k}\left(\cos\left(\frac{2\pi j}{k}\right)\cdot p_j-\sin\left(\frac{2\pi j}{k}\right)\cdot p_j^\perp \right) -\frac{1}{2}\overline{w}\cdot \mathbbm{n}(p_1)$ be a non-singular point of $\mathcal{P}_k$. Then the signed radius of curvature of $\mathcal{P}_k$ at $q$ is equal
\begin{align}\label{KappaCmws}
\rho_{\mathcal{P}_k}(q) = \frac{1}{k}\sum_{j=1}^{k} \rho_{\mathcal{H}}\left(p_j\right) - \frac{1}{2}\overline{w},
\end{align}
whereas $\rho_{\mathcal{H}}(p_1),\ldots \rho_{\mathcal{H}}(p_k)$ denote the signed radii of curvature of $\mathcal{H}$ at $p_1,\ldots, p_k$, respectively.

\end{thm}
\begin{proof}
Consider the parameterization of $\mathcal{H}$, then $p_j = \mathcal{H}\left(s+\frac{2\pi j}{k}\right)$. The signed radius of curvature of a hedgehog with a support function $h$ is known to be given by the following formula: $\rho(s)=h(s)+h''(s)$.  Therefore, one must notice that:
\begin{align*}
    \rho_{\mathcal{P}_k}(s) &= h_k(s) + h_k''(s)=\frac{1}{k}\sum_{j=1}^{k}h\Big(s+\frac{2\pi j}{k}\Big) - \frac{1}{2}\overline{w} +\frac{1}{k}\sum_{j=1}^{k}h''\Big(s+\frac{2\pi j}{k}\Big)\\&=\frac{1}{k}\sum_{j=1}^{k}\Bigg(h\Big(s+\frac{2\pi j}{k}\Big)+ h''\Big(s+\frac{2\pi j}{k}\Big)\Bigg) - \frac{1}{2}\overline{w}=\frac{1}{k}\sum_{j=1}^{k}\rho_{\mathcal{H}}\Big(s+\frac{2\pi j}{k}\Big)- \frac{1}{2}\overline{w},
\end{align*}
which concludes the proof.
\end{proof}

Leveraging the property that a hedgehog $\mathcal{H}$ is singular at points where its radius of curvature vanishes, and noting that a singularity at $\mathcal{H}(s)$, for $s \in [0, 2\pi]$, is a~cusp if and only if  $\det\left(\mathcal{H}''(s), \mathcal{H}'''(s)\right)\neq 0$, we can derive through direct calculation the following proposition:
\begin{prop}\label{SingCor}
Let $\mathcal{H}$ be a hedgehog of an average width equal to $\overline{w}$. Let $ \mathbbm{n}$ be a continuous unit normal vector field to $\mathcal{H}$.  Let $p_1,\ldots, p_k$ be an isogonal family of points in $\mathcal{H}$. Let $\rho_{\mathcal{H}}(p_1),\ldots \rho_{\mathcal{H}}(p_k)$ denote the signed radii of curvature of $\mathcal{H}$ at $p_1,\ldots, p_k$, respectively. The $k$th Order Preserving Set is singular at the point $\frac{1}{k}\sum_{j=1}^{k}\left(\cos\left(\frac{2\pi j}{k}\right)\cdot p_j-\sin\left(\frac{2\pi j}{k}\right)\cdot p_j^\perp  \right) -\frac{1}{2}\overline{w}\cdot \mathbbm{n}(p_1)$ if and only if the mean signed radius of curvature at the points $p_1,\ldots, p_k$ is equal to half the average width $\overline{w}$, i.e.,
\begin{align}\label{SingCondition}
\frac{1}{k}\sum_{j=1}^{k} \rho_{\mathcal{H}}\left(p_j\right) = \frac{1}{2}\overline{w}.
\end{align}
Furthermore, considering the differentiated parameterizations of the $k$th Order Preserving Set, a singular point $\mathcal{P}_k(s) $ is a cusp if and only if $\sum_{j=1}^{k}\rho'_{\mathcal{H}}\left(s+\frac{2\pi j}{k}\right) \neq 0$.
\end{prop}

Geometrically speaking, at an isogonal $k$-tuple of regular points $p_1,\ldots,p_k$ on $\mathcal H$, let us draw the
osculating circles strictly tangent to $\mathcal H$ and write
$r_j:=\rho_{\mathcal H}(p_j)$ for their signed radii. Set
$\overline{r}:=\frac{1}{k}\sum_{j=1}^k r_j$.
Then, the corresponding point in $\mathcal P_k$ is singular if and only if
$\overline{r}=\frac{\overline w}{2}$. Equivalently, $\mathcal P_k$ has a singularity precisely when the
isogonally averaged osculating radius equals half the average width of
the hedgehog, revealing an intriguing link between local curvature data
and the global quantity $\overline w$.


\section{The $k$th Order Midpoint Set} \label{sec4}

\noindent Consider a hedgehog $\mathcal{H}$, its Minkowski support function $h$ and some integer $k>2$. Consider the Fourier expansion of $h$, that is, 
\begin{align*}
h (s) = a_0 + \sum_{ n=1}^{\infty}a_n\cos(ns)+b_n\sin(ns).
\end{align*}
The Steiner point of $\mathcal{H}$ is $(a_1, b_1)$, see \eqref{eq:SteinerPoint}. Since the coefficients $a_1\cos(s),\ b_1\sin(s)$ account only for a translation of the hedgehog, we translate the whole system so that the Steiner point of $\mathcal{H}$ is in $(0,0)$.
Consider any $s\in[0,2\pi]$. The support function generates an equiangular $k$-gon with the distances from the origin to its sides given by $h(s),h\left(s+\frac{2\pi}{k}\right),\ldots , h\left(s+\frac{2\pi(k-1)}{k}\right)$. These sides are contained in support lines to $\mathcal{H}$ at $\mathcal{H}(s),\ldots,\mathcal{H}\left(s+\frac{2\pi(k-1)} {k}\right)$, respectively. A straightforward geometric observation reveals that our support function is $\frac{2\pi}{k}$-periodic if and only if every equiangular $k$-gon circumscribed about $\mathcal{H}$ is a regular $k$-gon with its center of mass in $(0,0)$. By Proposition ~\ref{kperiodic} we have that $\frac{2\pi}{k}$-periodicity is equivalent to the fact that the Fourier series of the support function retains only indices divisible by $k$ (including the index $0$ \textit{sic!}). Returning with our translation, we obtain the following result:
\begin{prop}\label{regularequiv}
    Every equiangular $k$-gon circumscribed about $\mathcal{H}$ is a regular \linebreak $k$-gon with its center of mass in the Steiner point of $\mathcal{H}$ if and only if the Fourier expansion of the support function of $\mathcal{H}$ retains only indices divisible by $k$ and its first harmonics, i.e.,
    $$h(s) = a_0 +a_1\cos(s)+b_1\sin(s)+ \sum_{k\mid n, \;n>1}a_n\cos(ns)+b_n\sin(ns).$$
\end{prop}
\begin{proof}
Consider a hedgehog $\mathcal{H}$ in the Cartesian plane, positioned such that its Steiner point coincides with the origin (without loss of generality).  
Assume that every equiangular $k$-gon circumscribed about $\mathcal{H}$ is a regular $k$-gon with its center of mass at the origin.

Let us consider any regular $k$-gon among those circumscribed about $\mathcal{H}$. Construct $k$ line segments from the center of mass of this $k$-gon (i.e., the origin) perpendicular to each of its $k$ sides. Owing to the regularity of the $k$-gon, these $k$ lines are of equal length.  
This common length is the distance from the origin to the support lines of $\mathcal{H}$, hence it equals the value of the support function of $\mathcal{H}$ at $k$ evenly spaced angles (these angles are $\frac{2\pi}{k}$ each). Since the choice of the regular $k$-gon is arbitrary, we have obtained that $h$, the support function of $\mathcal{H}$, is $\frac{2\pi}{k}$-periodic.

By invoking Proposition~\ref{kperiodic}, the support function of $\mathcal{H}$ takes the desired form
\[
h(s) = a_0 + \sum_{k\mid n, \;n>1} a_n \cos(ns) + b_n \sin(ns).
\]

Conversely, without loss of generality, assume the support function of $\mathcal{H}$ is
\[
h(s) = a_0 + \sum_{k\mid n, \;n>1} a_n \cos(ns) + b_n \sin(ns).
\]
By Proposition~\ref{kperiodic}, $h$ is $\frac{2\pi}{k}$-periodic.  
This implies that, for any equiangular $k$-gon circumscribed about $\mathcal{H}$, the distances from the origin to its $k$ sides are equal.

Denote the perpendicular line segments from the origin to the sides of the \linebreak $k$-gon by $\ell_1,\dots,\ell_k$ and let $w_i$ be the intersection point of $\ell_i$ with the side it meets.  
The~lines $\ell_i$ and $\ell_{i+1}$ (indices modulo $k$) form an isosceles triangle $t_i$ with vertex at the origin.  
Since the angle between consecutive $\ell_i$ and $\ell_{i+1}$ is constantly $\frac{2\pi}{k}$, the triangles $t_1,\dots,t_k$ are all congruent. Let $v_i$ be the vertex of the $k$-gon lying between sides supported on $\ell_i$ and $\ell_{i+1}$ (with $v_k$ between $\ell_k$ and $\ell_1$).  
Consider the triangles $T_i$ formed by the points $v_i$, $w_i$, and $w_{i+1}$ (with $T_k$ formed by $v_k$, $w_k$, and $w_1$). For a visual reference, see Figure \ref{Cap11}, where these geometric objects ($v_i$, $w_i$, $\ell_i$, $t_i$, and $T_i$) are illustrated.

\begin{figure}[h]
\centering
\includegraphics[scale=0.25]{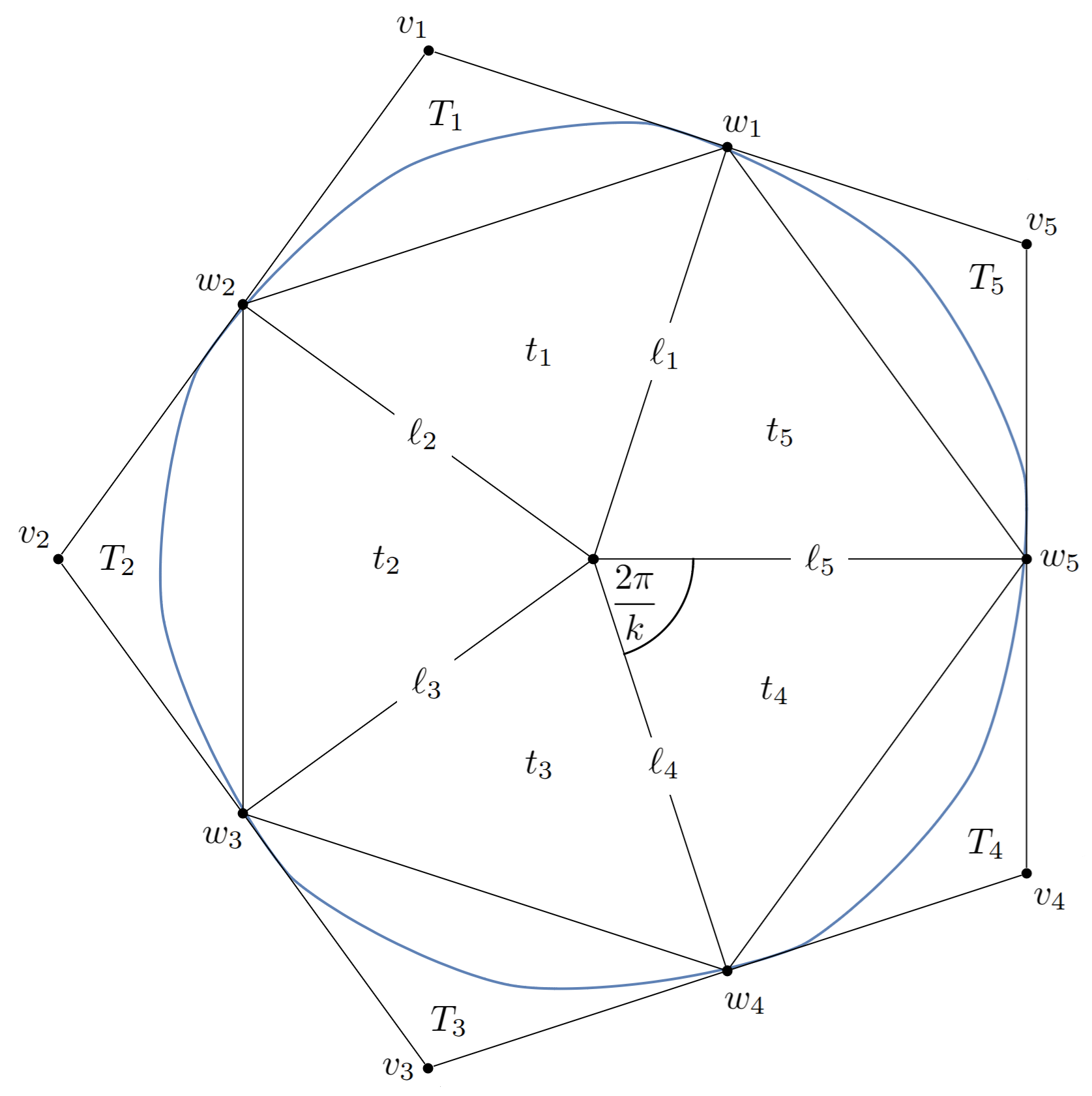}
\caption{A hedgehog $\mathcal{H}$ with a support function $h(s)=130+\sin (5s)+\sin (10s)$, a regular pentagon circumscribed about $\mathcal{H}$, and the Steiner point of $\mathcal{H}$ (centroid of the pentagon), with $v_i$, $w_i$, $\ell_i$, $t_i$, and $T_i$ indicated} 
\label{Cap11}
\end{figure} 

Each triangle $T_i$ has:
an angle $\pi - \frac{2\pi}{k}$ at $v_i$,
equal angles $\frac{\pi}{k}$ at $w_i$ and $w_{i+1}$,
a~base $w_i w_{i+1}$ of constant length (the base of the congruent triangle $t_i$).

Thus, the triangles $T_1,\dots,T_k$ are congruent.  
In particular, all sides of the $k$-gon are equal.  
Moreover, each $T_i$ is isosceles with apex $v_i$, so $w_{i+1}$ lies midway between $v_i$ and $v_{i+1}$ on the base.  
By congruence and the isosceles symmetry of the $T_i$, the~distances from the origin to the vertices $v_i$ are all equal.  

Therefore, the equiangular $k$-gon has equal sides and all vertices at the same distance from the origin. Hence, it is a regular $k$-gon with its center of mass at the~origin.
\end{proof}

The statement of Proposition \ref{regularequiv} is illustrated in Figure~\ref{Cap11}.

\begin{defn}
    Let $\mathcal{H}$ be a hedgehog and let $k>2$. Consider the family of isogonal (equiangular) $k$-gons circumscribed about $\mathcal{H}$. We call the family of the centers of mass of such $k$-gons the \textit{$k$th Order Midpoint Set} $\Omega_{\mathcal{H},k}$.
\end{defn} 
\begin{rem}
    In analogy with Remark~\ref{remCWMS}, the case $k = 2$ corresponds to “isogonal pairs” that can be interpreted as parallel pairs of points on the curve $\mathcal{H}$. The set of midpoints of all such parallel pairs is known in the literature by many names -- as the Wigner caustic, the middle hedgehog, the area evolute, and the defect of symmetry -- and has many applications, including isoperimetric-type problems (see~\cite{Berry, CraizerAreaEvolute, DFJSymmetryDefect, DMR1, DR1, DRZsecant, DZnew, DZ1, DZgb, GWZ1, JJR1, S5, S6, Z2}, and the references therein). It is worth noting that, although the Wigner caustic is typically a highly singular set, the $k$th Order Midpoint Set exhibits a different behavior: in most cases, it is smooth and free of singularities. Moreover, recent work has also explored the artistic aspects of the Wigner caustic and related sets, emphasizing their aesthetic appeal alongside their geometric properties (see~\cite{DPSZ-art}).
\end{rem}

\begin{ex}
    Consider a hedgehog $\mathcal{H}$ with a support function $h(s) = 10 + \sin(2s) + \cos(3s)+ \cos(4s) + \cos(5s) $ and its $3$rd Order Midpoint Set (see Figure~\ref{fig: fig6}).

\begin{figure}[h]
\centering
\includegraphics[scale=0.26]{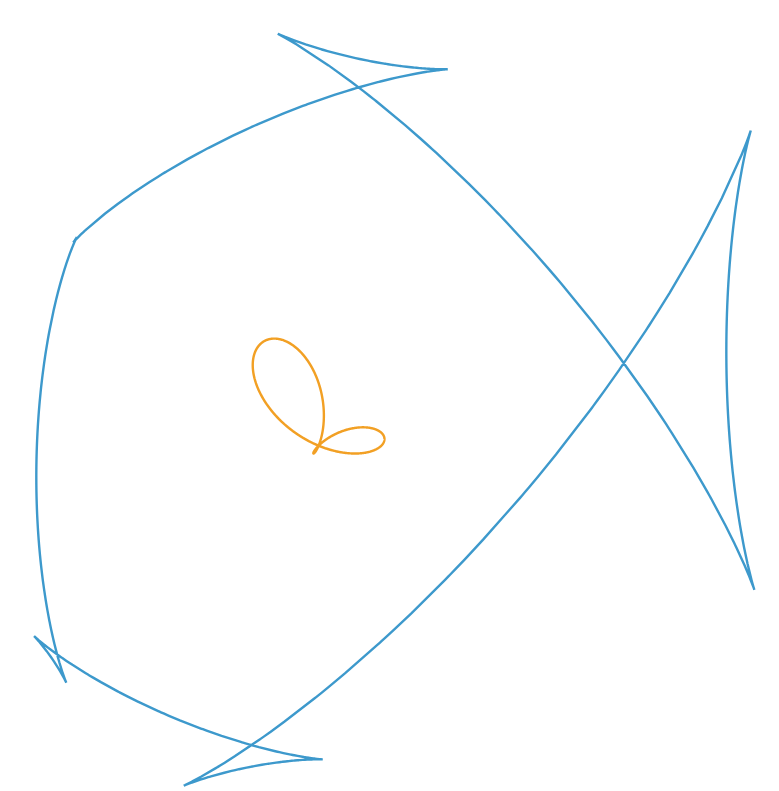}
\includegraphics[scale=0.26]{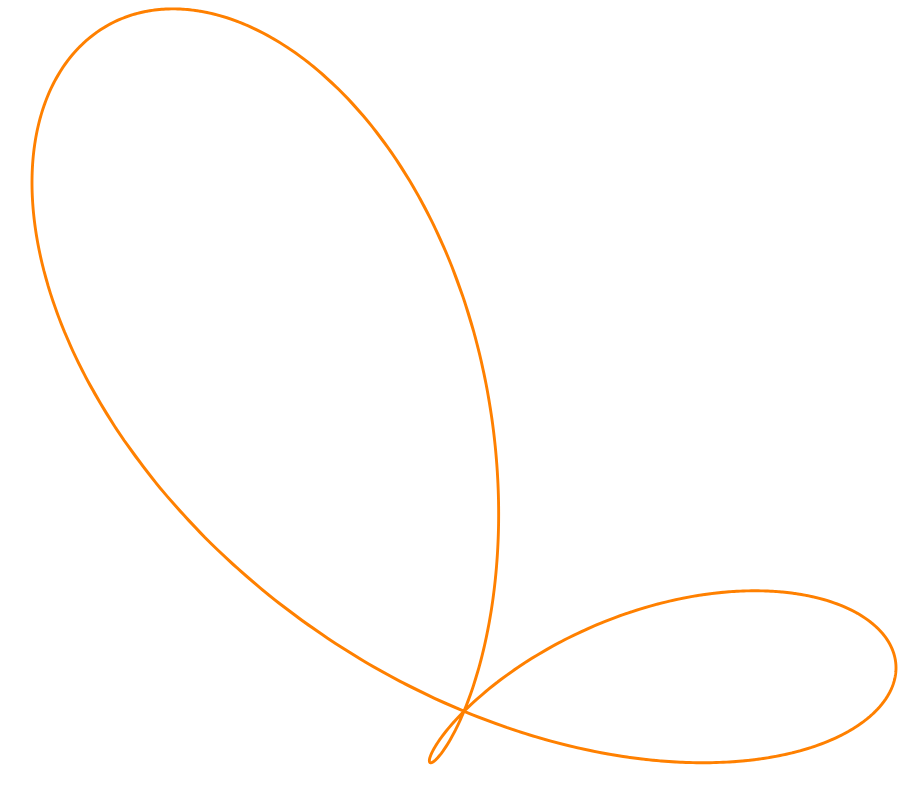}
\caption{On the left, the hedgehog $\mathcal{H}$ together with $\Omega_{\mathcal{H},3}$; on the right, $\Omega_{\mathcal{H},3}$ alone (magnified)}
\label{fig: fig6}
\end{figure} 

\end{ex}
Building on the geometric construction of equiangular $k$-gons described earlier, a straightforward analytical proof demonstrates that the $k$th Order Midpoint Set can be elegantly parameterized: 
\begin{prop}\label{prop:midpar}
Consider a hedgehog $\mathcal{H}$ and a natural number $k>2$. Then the $k$th Order Midpoint Set has the following parameterization:
\begin{equation}
\label{eq:midpar}
    \Omega_{\mathcal{H},k}(s) = \frac{2}{k}\sum_{j=1} ^k h\Big(s+\frac{2\pi j}{k}\Big)\Bigg(\cos\Big(s+\frac{2\pi j}{k}\Big),\sin\Big(s+\frac{2\pi j}{k}\Big)\Bigg)\quad\text{for}\quad s\in\left[0,\frac{2\pi}{k}\right].
\end{equation}
    
\end{prop}
\begin{proof}
Utilizing the nomenclature from the proof of Proposition \ref{regularequiv} (see also Figure~\ref{Cap11}), we introduce $\alpha$ -- the angle defined by $w_{j+1}$, the origin and $v_j$. Note that $\alpha \neq \frac{\pi}{2}$, by construction. One gets that $\cos(\alpha)=\frac{h\left(\frac{2\pi}{k}(j+1)+s\right)}{y_j}$, whereas $y_j$ is the distance from the origin to $v_j$. Similarly, $\cos\left(\frac{2\pi}{k}-\alpha\right)=\frac{h\left(\frac{2\pi}{k}j+s\right)}{y_j}$.  For simplicity sake, we denote $\sin\left(s+\frac{2\pi}{k}(j+1)\right)$ as $s_{j+1}$, $\cos\left(s+\frac{2\pi}{k}(j+1)\right)$ as $cs_{j+1}$ and likewise $h\left(s+\frac{2\pi}{k}(j+1)\right)$ as $h_{j+1}$. Using the cosine-difference identity, we obtain $\tan(\alpha) = \frac{\cos\left(\frac{2\pi}{k}-\alpha \right)-\cos\left(\alpha\right)\cdot\cos\left(\frac{2\pi}{k}\right)}{\cos\left(\alpha\right)\cdot\sin\left(\frac{2\pi}{k}\right)} = \frac{h_j - h_{j+1}\cos\left(\frac{2\pi}{k}\right)}{h_{j+1}\sin\left(\frac{2\pi}{k}\right)}$. By vector analysis, we see that the point $v_j$ is represented by the coordinates:
\begin{align*}
v_j &= h\left(s+\frac{2\pi}{k}(j+1)\right)\cdot\left(\cos\left(s+\frac{2\pi}{k}(j+1)\right),\sin\left(s+\frac{2\pi}{k}(j+1)\right)\right) \\&+ \tan(\alpha)\cdot h\left(s+\frac{2\pi}{k}(j+1)\right)\cdot\left(\cos\left(s+\frac{2\pi}{k}(j+1) - \frac{\pi}{2}\right),\sin\left(s+\frac{2\pi}{k}(j+1)- \frac{\pi}{2}\right)\right)\\&=
h_{j+1} \cdot(cs_{j+1},s_{j+1}) + \frac{h_j - h_{j+1}\cos\left(\frac{2\pi}{k}\right)}{\sin\left(\frac{2\pi}{k}\right)}\cdot\left(s_{j+1},-cs_{j+1}\right).
\end{align*}
We now arrive at a crucial computational point. Our goal is to find the centroid of the polygon defined by the vertices $v_1,\ldots,v_k$. To achieve this goal we recall the~Fourier series $h_j = \sum_{n=-\infty}^{\infty} c_n {\e}^{\i n\left(s+\frac{2\pi}{k}j\right)}$ and the fractions $\cos(s) = \frac{{\e}^{\i s}+{\e}^{-\i s}}{2}$ and $\sin(s) = \frac{{\e}^{\i s}-{\e}^{-\i s}}{2\i }$. Since all calculations are very much analogical, it suffices to provide a step-by-step evaluation of only the following sum:
\begin{align*}
    \sum_{j=0}^{k-1}h_{j+1}\cdot cs_{j+1} &= \sum_{j=0}^{k-1}\sum_{n=-\infty}^{\infty}c_n {\e}^{\i n\left(s+\frac{2\pi}{k}(j+1)\right)}\cdot \frac{{\e}^{\i \left(s+\frac{2\pi}{k}(j+1)\right)}+{\e}^{-\i \left(s+\frac{2\pi}{k}(j+1)\right)}}{2}\\ &=\sum_{n=-\infty}^{\infty}\frac{c_n}{2}{\e}^{\i ns}\Bigg(\sum_{j=0}^{k-1}{\e}^{\i s}\cdot {\e}^{\i (n+1)\left(\frac{2\pi}{k}(j+1)\right)}+\sum_{j=0}^{k-1}{\e}^{-\i s}\cdot {\e}^{\i (n-1)\left(\frac{2\pi}{k}(j+1)\right)}\Bigg).
\end{align*}
These operations are well-defined, since we are adding together finitely many convergent series. The internal expressions are finite geometric series, which can be conveniently reduced in the following fashion:
\begin{align*}
   \frac{1}{k}\sum_{j=0}^{k-1}h_{j+1}\cdot cs_{j+1} &=\sum_{ k \mid n+1}\frac{c_n}{2}{\e}^{\i (n+1)s} + \sum_{ k \mid n-1}\frac{c_n}{2}{\e}^{\i (n-1)s}.
\end{align*}
Similarly, one gets that:
$$ \frac{1}{k}\sum_{j=0}^{k-1}h_{j+1}\cdot s_{j+1} = \sum_{ k \mid n+1}\frac{c_n}{2i}{\e}^{\i (n+1)s} - \sum_{ k \mid n-1}\frac{c_n}{2\i }{\e}^{\i (n-1)s},$$
$$ \frac{1}{k} \sum_{j=0}^{k-1}h_j\cdot cs_{j+1} = \sum_{ k \mid n+1}\frac{c_n}{2}{\e}^{\i (n+1)s}{\e}^{\i \frac{2\pi}{k}} + \sum_{ k \mid n-1}\frac{c_n}{2}{\e}^{\i (n-1)s}{\e}^{-\i \frac{2\pi}{k}},$$
$$  \frac{1}{k}\sum_{j=0}^{k-1}h_j\cdot s_{j+1} = \sum_{ k \mid n+1}\frac{c_n}{2\i }{\e}^{\i (n+1)s}{\e}^{\i \frac{2\pi}{k}} - \sum_{ k \mid n-1}\frac{c_n}{2\i}{\e}^{\i (n-1)s}{\e}^{-\i \frac{2\pi}{k}}.$$ Combining these results, we find that:
\begin{align*}
    \frac{1}{k}\sum_{j=1}^{k}v_j &=\left(\sum_{ k \mid n+1}c_n{\e}^{\i (n+1)s} + \sum_{ k \mid n-1}c_n{\e}^{\i (n-1)s},\sum_{ k \mid n+1}\frac{c_n}{\i }{\e}^{\i (n+1)s} - \sum_{ k \mid n-1}\frac{c_n}{\i }{\e}^{\i (n-1)s}\right)\\
    &=\frac{2}{k}\Bigg(\sum_{j=0}^{k-1}h_{j+1}\cdot cs_{j+1},\sum_{j=0}^{k-1}h_{j+1}\cdot s_{j+1}\Bigg),
\end{align*} 
thus arriving at the proposed parameterization of the $k$th Order Midpoint Set.
Allowing the parameter $s$ to vary over the interval $[0, 2\pi]$ in the parameterization \eqref{eq:midpar} yields a $k$-fold covering of the $k$th Order Midpoint Set. Hence, it suffices to consider the restricted interval $s \in \left[0, \tfrac{2\pi}{k} \right]$ to describe the entire set without redundancy.
 \end{proof}   

From these considerations, one easily deduces the following proposition.
\begin{prop}\label{singmid}
    The $k$th Order Midpoint Set $\Omega_{\mathcal{H},k}$ of a given hedgehog $\mathcal{H}$ degenerates to a single point if and only if the support function of $\mathcal{H}$ has no Fourier harmonics with indices congruent to $\pm 1\pmod{k}$; equivalently, if and only if the~support function is equal to
    \begin{align}
    \label{eq:middlepointsupport}h(s) = \sum_{ k\nmid n+1\ \land \ k\nmid n-1} a_n\cos(ns) + b_n\sin(ns).
    \end{align}
\end{prop}

\begin{rem}
    A \textit{rotor} in a regular polygon is a convex closed curve that remains in continuous contact with the sides of the polygon while being fully rotated within it. In other words, it can be rotated through a full angle inside the polygon without losing contact with its boundary. 
    
    Meissner (\cite{Meissner}) demonstrated that such rotors in an $k$-sided regular polygon can be described analytically using their support functions:
    \begin{align*}
    h(s) =a_0+ \sum_{ k\mid n+1\ \vee \ k\mid n-1} a_n\cos(ns) + b_n\sin(ns).
    \end{align*}

    In other words, by Proposition~\ref{singmid}, the degeneration of the $k$th Order Midpoint Set to a single point is equivalent to $\mathcal{H}$ behaving, in a certain geometric sense, like the inverse harmonic counterpart of a rotor. A recent work (see \cite{Rochera3}) investigates a~dual problem to that of rotors: it characterizes curves along which a regular polygon can rotate while keeping all its vertices on the curve throughout the motion.
\end{rem}

Note that the support function of the hedgehog in Figure \ref{Cap11} is of the form \eqref{eq:middlepointsupport} -- therefore the corresponding $k$th Order Midpoint Set is a point. The opposite situation occurs with the hedgehog in Figure \ref{fig: fig6}.

To conclude this section, we establish a lemma giving an explicit expression for the oriented area of the $k$th Order Midpoint Set, which will serve as a crucial ingredient in the reasoning in the proof of the isoperimetric inequality in the next section.

\begin{lem}\label{lemOrientedAreaMiddlePoint}
    Let $k\in\mathbb{N}$, $k>2$, and let $\mathcal{H}$ be a hedgehog with support function  $$h(s)=a_0+\sum_{n=1}^{\infty}a_n\cos(ns)+b_n\sin(ns).$$ Then, the oriented area of the $k$th Order Midpoint Set $\Omega_{\mathcal{H},k}$ of $\mathcal{H}$ is given by the~following formula:
    \begin{align}
    \label{eq:MiddleOrientedArea}A_{\Omega_{\mathcal{H},k}} = \pi\sum_{n=1}^\infty n(a^2_{kn-1}+b^2_{kn-1}-a^2_{kn+1}-b^2_{kn+1}).
    \end{align}
\end{lem}
\begin{proof}
    For simplicity sake, we denote $h\left(s+\frac{2\pi}{k}j\right)$ as $h_{j}$ and $h'\left(s+\frac{2\pi}{k}j\right)$ as $h'_{j}$, for $j=1,\ldots,k$.   Let us notice that according to Proposition \ref{prop:midpar} we parameterize the $k$th Order Midpoint Set on the interval $[0,\frac{2\pi}{k}]$. Considering the parameterization given by equation \eqref{eq:midpar} and finding the oriented area using Green's~theorem, one can get the following expression for the oriented area of $\Omega_{\mathcal{H},k}$ traced $k$ times:
    \begin{align}
    \label{eqOmegaAreaGreen}kA_{\Omega_{\mathcal{H},k}} = \frac{2}{k^2}\int \limits_{0}^{2\pi}\sum_{m,j=1}^k h_jh_m'\sin\left(\frac{2\pi}{k}(m-j)\right)+h_jh_m\cos\left(\frac{2\pi}{k}(m-j)\right) \; \d s.
    \end{align}
    Note that parameterizing by $s$ over $[0,2\pi]$ gives us a $k$-tuple covering of the $k$th Order Midpoint Set.  This $k$-fold covering
    explains the prefactor $k$ on the left-hand side of
    \eqref{eqOmegaAreaGreen}. From Parseval's identity and the Fourier series of $h$ and $h'$, it follows that (in terms of Fourier coefficients of the expansion of $h$ -- the Minkowski support function of $\mathcal{H}$):
    \begin{align*}
        kA_{\Omega_{\mathcal{H},k}} &= \frac{2\pi}{k^2}\Bigg(\sum_{n=-\infty}^\infty\sum_{m,j} -n|c_n|^2\cdot \left({\e}^{\i \frac{2\pi}{k}(m-j)}-{\e}^{-\i \frac{2\pi}{k}(m-j)}\right) {\e}^{\i n\frac{2\pi}{k}(j-m)} \\&+ \sum_{n=-\infty}^\infty\sum_{m,j} |c_n|^2\cdot \left({\e}^{\i \frac{2\pi}{k}(m-j)}+{\e}^{-\i \frac{2\pi}{k}(m-j)}\right) {\e}^{\i n\frac{2\pi}{k}(j-m)}\Bigg)\\&=\frac{2\pi}{k^2}\Bigg(\sum_{n=-\infty}^\infty\sum_{m,j} -n|c_n|^2 {\e}^{\i \frac{2\pi}{k}(1-n)m}{\e}^{\i \frac{2\pi}{k}(n-1)j} +n|c_n|^2{\e}^{-\i \frac{2\pi}{k}(n+1)m}{\e}^{\i \frac{2\pi}{k}(n+1)j} \\&+  \sum_{n=-\infty}^\infty\sum_{m,j}|c_n|^2{\e}^{\i \frac{2\pi}{k}(1-n)m}{\e}^{\i \frac{2\pi}{k}(n-1)j}+|c_n|^2{\e}^{-\i \frac{2\pi}{k}(n+1)m}{\e}^{\i \frac{2\pi}{k}(n+1)j}\Bigg).
    \end{align*}
    Recalling the previously used geometric series method in the proof of Proposition \ref{prop:midpar}, we get that
    \begin{align*}
         A_{\Omega_{\mathcal{H},k}} &=  \frac{2\pi}{k^3}k^2\Bigg( \sum_{k\mid n-1} -n|c_n|^2 + \sum_{k\mid n+1} n|c_n|^2 +\sum_{k\mid n-1} |c_n|^2 + \sum_{k\mid n+1} |c_n|^2 \Bigg) \\&= \frac{2\pi}{k}\Bigg(  \sum_{k\mid n+1} (n+1)|c_n|^2 -\sum_{k\mid n-1} (n-1)|c_n|^2\Bigg)\\&=
         \frac{2\pi}{k}\Bigg(  \sum_{n=-\infty}^\infty kn|c_{kn-1}|^2 -\sum_{n=-\infty}^\infty kn|c_{kn+1}|^2\Bigg)  \\&= \frac{2\pi}{k}\frac{k}{2}\Bigg(  \sum_{n=1}^\infty na_{kn-1}^2+nb_{kn-1}^2 -na_{kn+1}^2-nb_{kn+1}^2\Bigg).
    \end{align*}
    This completes the proof.
\end{proof}

\begin{ex}
By Lemma \ref{lemOrientedAreaMiddlePoint}, the $k$th Order Midpoint Set may have positive, negative, or even zero oriented area, as illustrated in Figure~\ref{fig:PosNegZeroOriented} in the cases of $4$th Order Midpoint Sets.
 The support functions of the ovals corresponding to the sets in Figure \ref{fig:PosNegZeroOriented} are as follows:
\begin{itemize}
    \item $h(s)=81+\sin(3s)-\sin(5s)+\sin(7s)$ for Figure \ref{fig:MiddlePositive},
    \item $h(s)=105+\sin (3s)-2\sin(5s)+\sin(7s)$ for Figure \ref{fig:MiddleNegative},
    \item $h(s)=108+\sqrt{2}\sin (3s)-2\sin (5s)+\sin (7s)$ for Figure \ref{fig:MiddleZero}.
\end{itemize}

It is also worth noting that these 4th Order Midpoint Sets are not necessarily hedgehogs.
    
\end{ex}

\begin{figure}[h]
    \centering
    \begin{subfigure}[h]{0.28\textwidth}
        \centering
        \includegraphics[width=\textwidth]{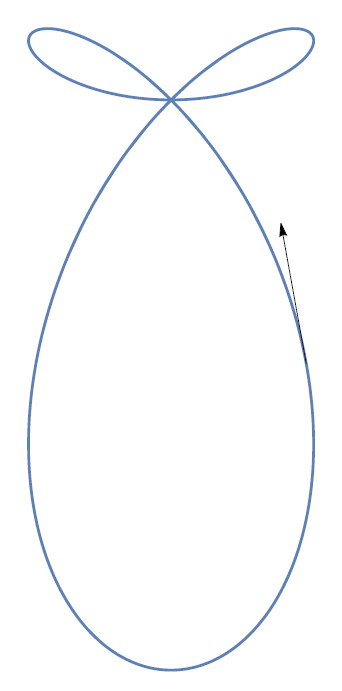}
        \caption{Positive}
        \label{fig:MiddlePositive}
    \end{subfigure}
    \hfill
    \begin{subfigure}[h]{0.36\textwidth}
        \centering
        \includegraphics[width=\textwidth]{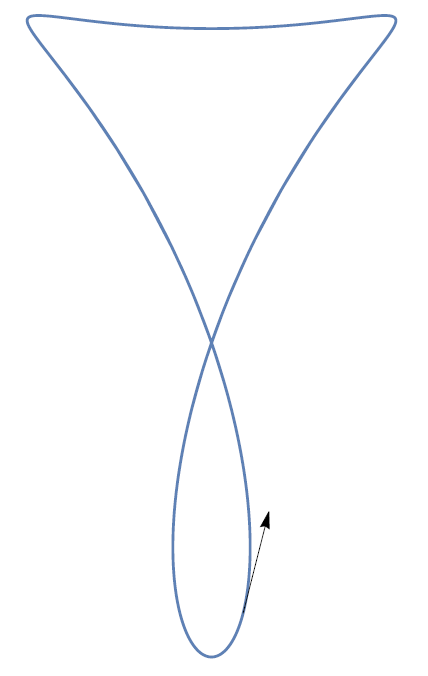}
        \caption{Negative}
        \label{fig:MiddleNegative}
    \end{subfigure}
    \hfill
    \begin{subfigure}[h]{0.28\textwidth}
        \centering
        \includegraphics[width=\textwidth]{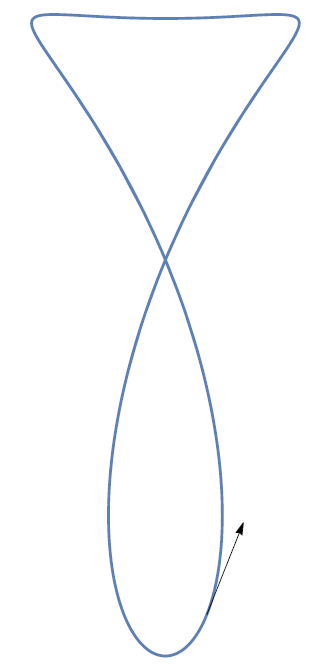}
        \caption{Zero}
        \label{fig:MiddleZero}
    \end{subfigure}
    \caption{$4$th Order Midpoint Sets with varying signs of their oriented areas (arrows indicate tangent vectors following the parameterization of the sets)}
    \label{fig:PosNegZeroOriented}
\end{figure}


\pagebreak

\section{Isoperimetric-Type Inequalities and Stability Results}\label{sec5}
\subsection{Isoperimetric-type inequalities}\label{ssec51}
\noindent In this subsection, we establish a family of isoperimetric-type inequalities associated with the sets introduced in earlier sections -- notably, the $k$th Order Preserving Set and the $k$th Order Midpoint Set.

\begin{thm}\label{ThmIsoEq1}
Let $k\in \mathbb{N}, k>2$, $\mathcal{O}$ be an oval, and $\mathcal{P}_k$ be its $k$th Order Preserving Set. Then
\begin{align}\label{IsoEq}
L_{\mathcal{O}}^2 - 4\pi A_{\mathcal{O}} \geqslant 4\pi |A_{\mathcal{P}_k}| ,
\end{align}
where $L_{\mathcal{O}}, A_{\mathcal{O}}, A_{\mathcal{P}_k}$ are the length of $\mathcal{O}$, the area bounded by $\mathcal{O}$ and the oriented area of the $k$th Order Preserving Set $\mathcal{P}_k$, respectively. Equality holds if and only if every equiangular $k$-gon circumscribed about $\mathcal{O}$ is a regular $k$-gon with its center of mass at the Steiner point of $\mathcal{O}$.
\end{thm}
\begin{proof}
Let us consider $A_{\mathcal{P}_k}$ -- the oriented area of $\mathcal{P}_k$. By Proposition \ref{Preservingsuppth}, the set $\mathcal{P}_k$ is a hedgehog with support function $h_k$ and by equation \eqref{HAreaFourier}, we get
 \begin{align}\label{eq1}
 A_{\mathcal{P}_k} &= \frac{1}{2}\int \limits_{0}^{2\pi} \left(h_k(s)^2 - h_k'^2(s)\right)\; \d s,
 \end{align}
 where $h$ is the Minkowski support function of $\mathcal{O}$. Next, applying Parseval's theorem, we obtain:
 \begin{align*}
 A_{\mathcal{P}_k} &= \frac{\pi}{2}\sum_{k \mid n, \; n >1}^{\infty} (1-n^2)(a_{n}^2 + b_{n}^2),\\
 4\pi |A_{\mathcal{P}_k}| &= 2\pi^2\sum_{k \mid n, \; n >1}^{\infty} (n^2-1)(a_{n}^2 + b_{n}^2).
 \end{align*}
Now, considering the isoperimetric deficit (see equation \eqref{isdef}), we observe:
{\begin{align}\label{isdef2}
    L_{\mathcal{O}}^2 - 4\pi A_{\mathcal{O}} &= 4\pi |A_{\mathcal{P}_k}| + 2\pi^2\sum_{k \nmid n,\ n>1} (n^2-1)(a_{n}^2 + b_{n}^2) \\ \nonumber &\geqslant 4\pi |A_{\mathcal{P}_k}| + 6\pi^2\sum_{k \nmid n,\ n>1}(a_{n}^2 + b_{n}^2)\geqslant 4\pi |A_{\mathcal{P}_k}|.
\end{align}}
It is evident that the equality case in inequality \eqref{isdef2} holds if and only if \linebreak $\sum_{k \nmid n} (n^2-1)(a_{n}^2 + b_{n}^2) = 0$, which means that the support function omits indices not divisible by $k$ and greater than $1$. By Proposition \ref{regularequiv}, we complete the proof.
\end{proof}

\begin{thm}\label{ThmIsoEq2}
Let $k\in \mathbb{N}, k>2$, $\mathcal{O}$ be an oval, $\mathcal{P}_k$ be its $k$th Order Preserving Set, and $\Omega_{\mathcal{O},k}$ its $k$th Order Midpoint Set. Then
\begin{align}\label{IsoEq2}
L_{\mathcal{O}}^2 - 4\pi A_{\mathcal{O}} \geqslant 4\pi |A_{\mathcal{P}_k}| + 2\pi |A_{\Omega_{\mathcal{O},k}}|,
\end{align}
where $L_{\mathcal{O}}, A_{\mathcal{O}}, A_{\mathcal{P}_k},A_{\Omega_{\mathcal{O},k}}$ are the length of $\mathcal{O}$, the area bounded by $\mathcal{O}$, the oriented area of the $k$th Order Preserving Set $\mathcal{P}_k$ and the oriented area of the $k$th Order Midpoint Set $\Omega_{\mathcal{O},k}$, respectively. Equality holds if and only if every equiangular \linebreak $k$-gon circumscribed about $\mathcal{O}$ is a regular $k$-gon with its center of mass at the Steiner point of $\mathcal{O}$.
\end{thm}
\begin{proof}
By Lemma \ref{lemOrientedAreaMiddlePoint}, the oriented area of the $k$th Order Midpoint Set is
$$A_{\Omega_{\mathcal{O},k}} = \pi \sum_{n=1}^\infty n(a^2_{kn-1}+b^2_{kn-1}-a^2_{kn+1}-b^2_{kn+1}).$$
    Following the previous proof, we obtain:
    {\small\begin{align*}
        & L_{\mathcal{O}}^2 - 4\pi A_{\mathcal{O}} - 4\pi |A_{\mathcal{P}_k}| - 2\pi|A_{\Omega_{\mathcal{O},k}}| \\= \ &2\pi^2\left(\sum_{k \nmid n} (n^2-1)(a_{n}^2 + b_{n}^2) \mp \sum_{n=1}^\infty n\left(a^2_{kn-1}+b^2_{kn-1}\right)\pm\sum_{n=1}^\infty n\left(a^2_{kn+1}+b^2_{kn+1}\right)\right)\\=\ & 2\pi^2\Bigg(\sum_{k \nmid n,\ n\not \equiv_k\pm 1} (n^2-1)(a_{n}^2 + b_{n}^2) +\sum_{k \nmid n,\ n\equiv_k\pm 1} (n^2-1)(a_{n}^2 + b_{n}^2)\\ & \mp \sum_{n=1}^\infty n\left(a^2_{kn-1}+b^2_{kn-1}\right)\pm\sum_{n=1}^\infty n\left(a^2_{kn+1}+b^2_{kn+1}\right)\Bigg)\\=\ & 2\pi^2\Bigg(\sum_{k \nmid n,\ n\not \equiv_k\pm 1} (n^2-1)(a_{n}^2 + b_{n}^2) +\sum_{n=kj\pm 1,\ j\in \mathbb{N}_+} (n^2-1)(a_{n}^2 + b_{n}^2)\\ & \mp \sum_{n=1}^\infty n\left(a^2_{kn-1}+b^2_{kn-1}\right)\pm\sum_{n=1}^\infty n\left(a^2_{kn+1}+b^2_{kn+1}\right)\Bigg)\\=\ & 2\pi^2\Bigg(\sum_{k \nmid n,\ n\not \equiv_k\pm 1} (n^2-1)(a_{n}^2 + b_{n}^2) +\sum_{n=1} ^\infty((kn\pm 1)^2-1)(a_{kn\pm 1}^2 + b_{kn\pm 1}^2)\\ & \mp \sum_{n=1}^\infty n\left(a^2_{kn-1}+b^2_{kn-1}\right)\pm\sum_{n=1}^\infty n\left(a^2_{kn+1}+b^2_{kn+1}\right)\Bigg)\\=\ & 2\pi^2\Bigg(\sum_{k \nmid n,\ n\not \equiv_k\pm 1} (n^2-1)(a_{n}^2 + b_{n}^2) \\ & + \sum_{n=1}^\infty ((kn-1)^2\mp n-1)\left(a^2_{kn-1}+b^2_{kn-1}\right)+\sum_{n=1}^\infty ((kn+1)^2\pm n-1)\left(a^2_{kn+1}+b^2_{kn+1}\right)\Bigg) \geqslant 0.\end{align*}}
        
        \noindent The preceding inequality is true because $k>1$. Therefore, this concludes the proof of inequality \eqref{IsoEq2}. Since $n^2-1>0$, $(kn+1)^2\pm n-1>0$ and $(kn-1)^2\mp n-1 > 0$, the~equality case occurs if and only if all the coefficients  $a_n, b_n$, where $k\nmid n,\ n\not \equiv_k\pm 1$, and $n \equiv_k\pm 1$, are zero. This leaves only the coefficients $a_n,b_n$ such that $k \mid n$. Equivalently, the support function of the oval $\mathcal{O}$ is of the following form:
    $$h(s) = a_0 +a_1\cos(s) + b_1\sin(s)+ \sum_{k\mid n, \;n>1}a_n\cos(ns) + b_n\sin(ns),$$
    i.e., every equiangular $k$-gon circumscribed about $\mathcal{O}$ is a regular $k$-gon with its center of mass at the Steiner point of $\mathcal{O}$ (see Proposition \ref{regularequiv}).
\end{proof}

   \subsection{Stability results}\label{ssec52}
In addition to proving the inequalities themselves, we investigate the stability of selected cases, analyzing how deviations from equality reflect the geometric and analytic perturbations of the input oval.  
A \textit{$2$-dimensional convex body} in $\mathbb{R}^2$ is a bounded convex subset in $\mathbb{R}^2$ which is closed and has non-empty interior. By $\mathcal{C}_2$ we denote the class of all $2$-dimensional convex bodies.

An inequality in convex geometry can be written as
\begin{align}\label{IneqConvexGeometry} 
\Phi(K)\geqslant 0,
\end{align}
where $\Phi:\mathcal{C}_2\to\mathbb{R}$ is a function and inequality \eqref{IneqConvexGeometry} holds for all $K$ in $\mathcal{C}_2$. Let $\mathcal{C}_{2,{\Phi}}$ be a subset of $\mathcal{C}_2$ for which the equality in \eqref{IneqConvexGeometry} holds. 

Let $L_{\partial K}$ (respectively $A_{\partial K}$) denote the length of the boundary of $K$ (respectively the area enclosed by $\partial K$, i.e., the area of $K$).

In this section, we study stability properties associated with (\ref{IneqConvexGeometry}). We ask if $K$ must be close to a member of $\mathcal{C}_{2,{\Phi}}$ whenever $\Phi(K)$ is close to zero. If $d:\mathcal{C}_2\times\mathcal{C}_2\to\mathbb{R}$ satisfies the following conditions:
\begin{enumerate}[(i)]
\item $d(K,L)\geqslant 0$ for all $K,L\in\mathcal{C}_2$,
\item $d(K,L)=0$ if and only if $K=L$,
\end{enumerate}
then $d$ measures the deviation between two convex bodies.

If $\Phi, \mathcal{C}_{2,{\Phi}}$ and $d$ are given, then the \textit{stability problem} associated with (\ref{IneqConvexGeometry}) is as follows.

\begin{prob}Find positive constants $c,\alpha$ such that for each $K\in\mathcal{C}_2$, there exists $N\in \mathcal{C}_{2,{\Phi}}$, such that 
\begin{align}\label{StabilityIneq}
\Phi(K)\geqslant cd^{\alpha}(K,N).
\end{align}
\end{prob}

Let $h_{\partial K}$ and $h_{\partial N}$ be support functions of convex bodies $K$ and $N$, respectively. To measure the deviation between $K$ and $N$ one can use the \textit{Hausdorff distance}, $d_{\infty}$, and the measure that corresponds to the $L_2$-metric in the function space, $d_2$, which are given by:
\begin{align}\label{HausdorffDistance}
d_{\infty}(K,N)&=\max_{s}\Big|h_{\partial K}(s)-h_{\partial N}(s)\Big|,\\
\label{LTwoDistance}
d_2(K,N)&=\left(\int_0^{2\pi}\Big|h_{\partial K}(s)-h_{\partial N}(s)\Big|^2\d s\right)^{\frac{1}{2}}.
\end{align}

It is immediate that $d_{\infty}(K,N)=0$ (or $d_2(K,N)=0$) if and only if $K=N$.

\begin{defn}
    Consider a 2-dimensional convex body $\mathcal{C}$ and $k>2$. We define its \textit{$k$-central symmetral} $\mathcal{S}_{k,\mathcal{C}}$ as the following Minkowski sum: 
    $$\mathcal{S}_{k,\mathcal{C}}:= \frac{1}{k}\left( \mathcal{C} + \operatorname{rot}\left(\mathcal{C},\frac{2\pi}{k}\right) + \cdots + \operatorname{rot}\left(\mathcal{C},\frac{2\pi(k-1)}{k}\right)\right),$$
    where $\operatorname{rot}\left(\mathcal{C},\alpha\right)$ is the counterclockwise rotation of $\mathcal{C}$ about the origin by an angle $\alpha$.
\end{defn}

Note that $k$-central symmetral $\mathcal{S}_{k,\mathcal{C}}$ is the Minkowski average of $k$ rotated copies of $\mathcal C$ about the origin; 
it can be viewed as a discrete rotational symmetrization. We recall that the classical \emph{central symmetral} is
\[
\Delta\mathcal{C} = \tfrac12\big(\mathcal{C}+(-\mathcal{C})\big),
\]
i.e., the Minkowski symmetrization with respect to the origin (\cite{SchneiderBook}).

    The $k$-central symmetral $\mathcal{S}_{k,\mathcal{C}}$ of a given convex body $\mathcal{C}$ is convex, as it is the~Minkowski sum of convex bodies.

    Consider a convex body $\mathcal{C}$ and its $k$th Order Preserving Set $\mathcal{P}_k$. One may check that $\mathcal{P}_k + D_{\mathcal{C}}$ is convex as the $k$-central symmetral of $\mathcal{C}$, where $D_{\mathcal{C}}$ is the Steiner disk of $\mathcal{C}$.
    Furthermore, it is evident that the Fourier expansion of the support function of $\mathcal{C}$ retains only indices divisible by $k$ if and only if
    $$\mathcal{C} = \mathcal{S}_{k,\mathcal{C}} = \mathcal{P}_k + D_{\mathcal{C}}.$$

\begin{rem}
Let $h$ be the support function of a hedgehog $\mathcal{H}$.
Define the \linebreak \textit{$k$-directional averaging operator}
$$\mathcal A_k[h](s):=\frac{1}{k}\sum_{j=0}^{k-1} h\!\left(s+\frac{2\pi j}{k}\right).$$
Our $k$th Order Preserving Set, $\mathcal P_k$, is the hedgehog with support function
$$h_{\mathcal P_k}=\mathcal A_k[h]-\frac{1}{2}\overline{w},$$
equivalently
$$h_{\mathcal P_k+D_{\mathcal{H}}}=\mathcal A_k[h],$$
i.e., the Minkowski sum $\mathcal P_k+D_{\mathcal{H}}$ restores the constant term.

The generalized $k$-width of Ou and Pan (\cite{OuPan}) is
\[
  w_k(s)=\sum_{j=0}^{k-1} h\left(s+\frac{2\pi j}{k}\right).
\]
So $\mathcal A_k[h]=\frac{1}{k}w_k$. Kwong in \cite{kwong} uses the closely related
operator
$$
  T_k[h](s):=\frac{1}{k}\sum_{m=1}^{k} h\left(s+\frac{(2m-1)\pi}{k}\right)
                 =\mathcal A_k[h]\left(s+\frac{\pi}{k}\right),
$$
whose Fourier coefficients are intimately connected with the coefficients in the Fourier series for $h_{\mathcal{P}_k}$. Namely (see \eqref{eqFourierkOPS}):
\begin{align*}
    h_k(s) &= \sum_{n=kq,\ q> 0}^\infty a_n\cos(ns) + b_n\sin(ns),\\ 
    T_k[h](s) &=\sum_{n=kq,\ q\geqslant 0}^{\infty}(-1)^q\left(a_n\cos(ns) + b_n\sin(ns)\right).
\end{align*}
In particular,
$T_k$ preserves the constant term. This operator underlies higher-order Chernoff-type inequalities for ovals, with equality holding only for circles. One of the simplest of those inequalities in \cite{kwong} is
\begin{align*}
A_{\mathcal{O}}\leqslant\frac{1}{k}\int_{0}^{\pi/k}
  w_k(s)\,w_k\left(s+\frac{\pi}{k}\right)\,\d s,
\end{align*}
which is one of the main results in \cite{OuPan}.

In summary of this remark, $\mathcal P_k$ encodes in some way $k$-directional averaging as
$w_k$ and $T_k$, but in a zero-mean normalization: we remove
the constant Fourier term.
Consequently, our stability bounds are stated for the body $\mathcal P_k+D_{\mathcal{H}}$ (see Theorem \ref{Thmstab1} and \ref{Thmstab2}), whose support function is exactly
$\mathcal A_k[h]$, making it directly comparable with the
Ou--Pan and Kwong framework. In addition, our main inequalities point in the \textit{ same} direction as the classical isoperimetric inequality (it yields
a lower bound for the isoperimetric deficit), in contrast to Chernoff-type
inequalities, which are examples of the so-called \textit{reversed isoperimetric-type inequalities} bounds
(upper bounds on the area, in this case in terms of average widths).
\end{rem}

\begin{thm}\label{Thmstab1}
Let $k\in \mathbb{N},\ k>2$, $\mathcal{O}$ be an oval, $\mathcal{P}_k$ be its $k$th Order Preserving Set, and $D_{\mathcal{O}}$ its Steiner disk. Then
\begin{align*}
     L_{\mathcal{O}}^2 - 4\pi A_{\mathcal{O}}-4\pi |A_{\mathcal{P}_k}|\geqslant\frac{8\pi^2k}{2 \pi \cot\left(\frac{\pi}{k}\right) +k}d^2_\infty(\mathcal{O},\mathcal{P}_k+ D_{\mathcal{O}}).
\end{align*}
\end{thm}
\begin{proof}
     Given the support function $h$ of $\mathcal{O}$, let us consider the following:
     \begin{align*}       
     d_\infty(\mathcal{O},\mathcal{P}_k+ D_{\mathcal{O}})&=\max_s\Big|h(s)-\Big(h_k(s) + h_{D_{\mathcal{O}}}(s)\Big)\Big|\\ &=\max_s\Big|\sum_{k \nmid n,\; n>1}a_n\cos(ns)+b_n\sin(ns)\Big|\\
     &\leqslant \sum_{k \nmid n, \;n>1}\max_s|a_n\cos(ns)+b_n\sin(ns) |\\
     &\leqslant  \sum_{k \nmid n,\; n>1}\sqrt{a_n^2+b_n^2 }=\sum_{k \nmid n,\; n>1}\sqrt{\frac{1}{n^2-1}}\sqrt{(n^2-1)(a_n^2+b_n^2) } \\
     &\leqslant\sqrt{\sum_{k \nmid n,\; n>1}\frac{1}{n^2-1}} \cdot\sqrt{\sum_{k \nmid n,\; n>1}(n^2-1)(a_n^2+b_n^2)}\\
     &= \sqrt{\frac{1}{2}\sum_{k \nmid n,\; n>1}\frac{1}{n-1} - \frac{1}{n+1}}\cdot\sqrt{\sum_{k \nmid n,\; n>1}(n^2-1)(a_n^2+b_n^2)}.\end{align*} Using the well-known formula $\pi\cot\left(\pi x\right)=\frac{1}{x}+\sum_{n=1}^{\infty}\left(\frac{1}{x+n}+\frac{1}{x-n}\right)$ for $x\in\mathbb{R}\setminus\mathbb{Z}$ (e.g., see \cite{thebook}),  one obtains that:
      \begin{align*}       
     d_\infty(\mathcal{O},\mathcal{P}_k+ D_{\mathcal{O}})\leqslant \sqrt{\frac{1}{2} \left( \frac{\pi \cot\left(\frac{\pi}{k}\right)}{k} - 1 \right) + \frac{3}{4}}\cdot\sqrt{\sum_{k \nmid n,\; n>1}(n^2-1)(a_n^2+b_n^2)}.
\end{align*}
Which in turn gives us:
$$\frac{8\pi^2k}{2 \pi \cot\left(\frac{\pi}{k}\right) +k}d^2_\infty(\mathcal{O},\mathcal{P}_k+ D_{\mathcal{O}}) \leqslant 2\pi^2\sum_{k \nmid n,\; n>1}(n^2-1)(a_n^2+b_n^2).$$
By equality~\eqref{isdef2}, we complete the proof.
\end{proof}

\begin{thm}\label{Thmstab2}
Let $k\in \mathbb{N},\ k>2$, $\mathcal{O}$ be an oval, $\mathcal{P}_k$ be its $k$th Order Preserving Set, and $D_{\mathcal{O}}$ its Steiner disk. Then
\begin{align*}
     L_{\mathcal{O}}^2 - 4\pi A_{\mathcal{O}}-4\pi |A_{\mathcal{P}_k}|\geqslant 6\pi d^2_2(\mathcal{O},\mathcal{P}_k+ D_{\mathcal{O}}).
\end{align*}
\end{thm}
\begin{proof}
     Given $h$, the support function of $\mathcal{O}$, let us consider:
     \begin{align*}       
     d^2_2(\mathcal{O},\mathcal{P}_k+ D_{\mathcal{O}})&=\int_0^{2\pi}\Big|h(s)-\Big(h_k(s) + h_{D_{\mathcal{O}}}(s)\Big)\Big|^2\d s\\ &=
     \int_0^{2\pi}\Big|\sum_{k \nmid n,\; n>1}a_n\cos(ns)+b_n\sin(ns)\Big|^2\d s \\
     &= \pi\sum_{k \nmid n,\; n>1}a_n^2+b_n^2 \leqslant \frac{1}{6\pi}2\pi^2\sum_{k \nmid n, n>1 }(n^2-1)(a_n^2+b_n^2).
\end{align*}

Equality~\eqref{isdef2} now implies the result, completing the proof.
\end{proof}

As $k$ tends towards infinity, the $k$th Order Preserving Set, $\mathcal{P}_k$, tends to the origin (see Remark \ref{remSteiBall}). Therefore, as $k$ tends towards infinity, we obtain the following corollary of Theorems \ref{Thmstab1} and \ref{Thmstab2}.

\begin{cor}\label{cor:stability}
    Let $\mathcal{O}$ be an oval and $D_{\mathcal{O}}$ its Steiner disk. Then
\begin{align*}
     L_{\mathcal{O}}^2 - 4\pi A_{\mathcal{O}} &\geqslant\frac{8}{3}\pi^2d^2_\infty(\mathcal{O},D_{\mathcal{O}}),\\ 
     L_{\mathcal{O}}^2 - 4\pi A_{\mathcal{O}}&\geqslant 6\pi d^2_2(\mathcal{O},D_{\mathcal{O}}).
\end{align*}
\end{cor}

The inequalities in Corollary \ref{cor:stability} are classical -- see, for instance, inequalities (4.3.1) and (4.3.2) in \cite{G4}.

Theorems~\ref{Thmstab1} and \ref{Thmstab2} quantify the stability: the isoperimetric
deficit controls the deviation of an oval \(\mathcal O\) from the class
\(\mathcal P_k+D_{\mathcal{O}}\) when this deviation is measured by \(d_\infty\) or \(d_2\).
In particular, if the deficit $L_{\mathcal{O}}^2 - 4\pi A_{\mathcal{O}}-4\pi |A_{\mathcal{P}_k}|$ is small, then \(\mathcal O\) is close (in either
distance) to some element of the equality case in Theorem \ref{ThmIsoEq1}. Conversely, vanishing deficit
forces \(d_p(\mathcal O,\mathcal P_k+D_{\mathcal{O}})=0\) for \(p\in\{2,\infty\}\).
Letting \(k\to\infty\) yields Corollary~\ref{cor:stability}, in which the
reference set reduces to the Steiner disk \(D_{\mathcal{O}}\) and the bounds acquire the
limiting constants \(\tfrac{8}{3}\pi^2\) and \(6\pi\). This completes the
stability part of the section and ties the general \(k\)-fold framework back to
the classical planar case.

\begin{ex}
To illustrate our stability estimates in practice, we now examine a~concrete oval defined by an explicit support function. In this way we
can compute all relevant quantities (like lengths, areas, and distances), and directly check the~isoperimetric inequality and the quantitative
bounds. Fix
\begin{align*}
h(s)=137 + 21\cos(2s) + \sin(5s)+\cos(6s)-\frac{1}{3}\sin(9s)+\frac{1}{3}\sin(10s)
\end{align*}
as a support function of an oval $\mathcal{O}$. The Steiner point of $\mathcal{O}$ is at the origin ($n=1$ harmonics vanish). Take $k=5$. Then, the support function of $5$th Order Preserving Set is
\begin{align}
\label{eq:exh}    h_5(s)=\sin(5s)+\frac{1}{3}\sin(10s).
\end{align}

In Figure~\ref{fig:exOPk} we present the oval $\mathcal O$ together with
its $5$-central symmetral and the associated $5$th order preserving set
$\mathcal P_5$. For clarity, in Figure~\ref{fig:exP5} we provide a~magnified view of $\mathcal P_5$ itself.

\pagebreak 

\begin{figure}[h]
    \centering
    \includegraphics[width=1\linewidth]{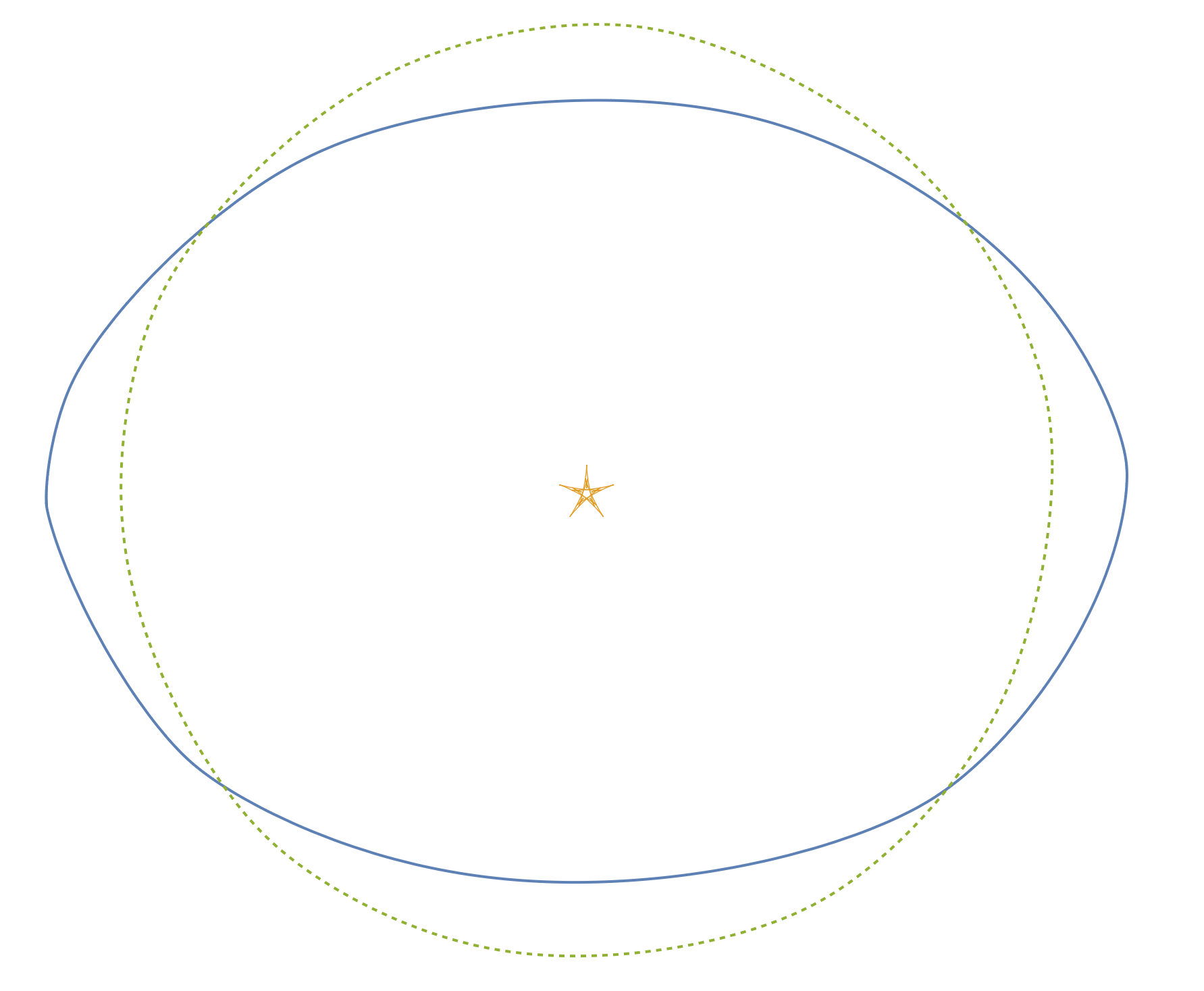}
    \caption{An oval $\mathcal{O}$ (solid line) with support function \eqref{eq:exh} together with its $5$-central symmetral (dashed line) and its $5$th Order Preserving Set (pentagonal curve)}
    \label{fig:exOPk}
\end{figure}

\begin{figure}[h]
    \centering
    \includegraphics[width=0.5\linewidth]{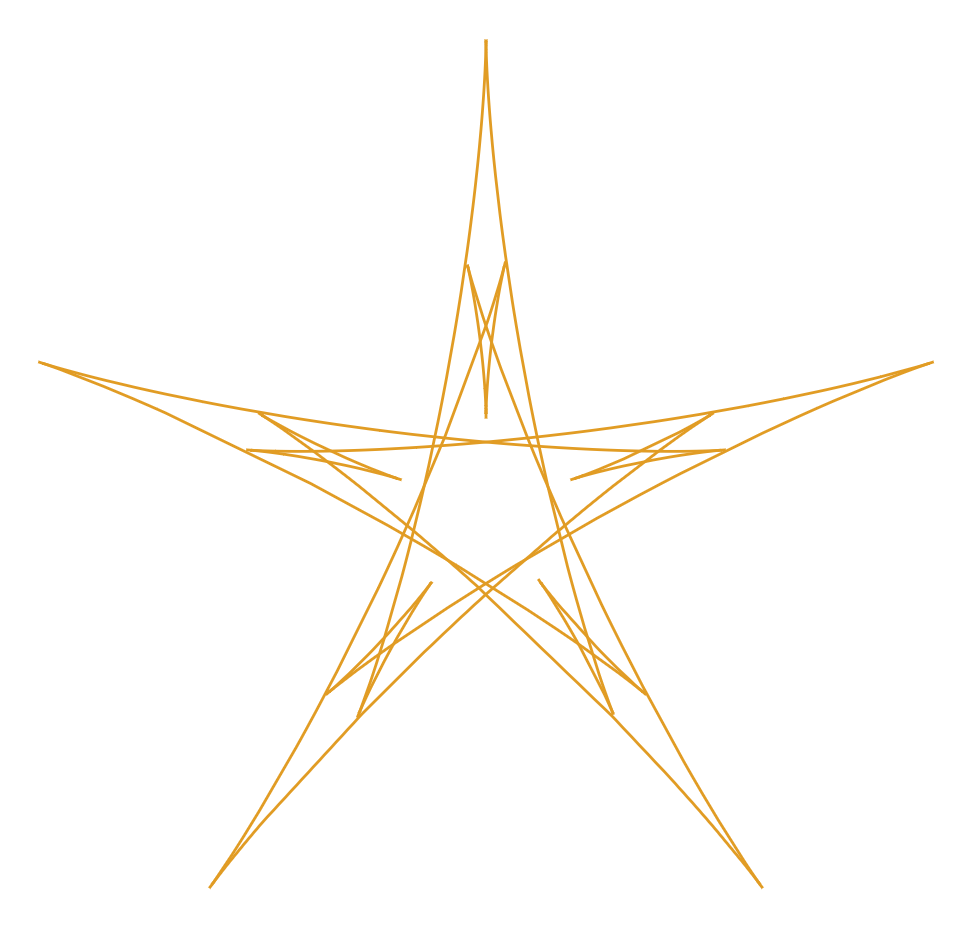}
    \caption{Magnified $5$th Order Preserving Set from Figure \ref{fig:exOPk}}
    \label{fig:exP5}
\end{figure}

\pagebreak

Then, by direct calculations:
\begin{align*}
    L_{\mathcal{O}}&=274\pi,\\
    A_{\mathcal{O}}&=\frac{325\, 225}{18}\pi,\\
    A_{\mathcal{P}_5}&=-\frac{35}{2}\pi,\\
    \Delta_5(\mathcal O):=L_{\mathcal{O}}^2-4\pi A_{\mathcal{O}}-4\pi\left|A_{\mathcal{P}_5}\right|&=\frac{24\, 604}{9}\pi^2\approx 26\, 981.3,\\ 
    d_{\infty}\left(\mathcal{O},\mathcal{P}_k+D_{\mathcal{O}}\right)&=\frac{67}{3},\\ 
    \frac{8\pi^2\cdot 5}{2\pi\cot\frac{\pi}{5}+5}d_{\infty}^2\left(\mathcal{O},\mathcal{P}_k+D_{\mathcal{O}}\right)&=\frac{179\, 560\pi^2}{9\left(5+2\sqrt{1+\frac{2}{\sqrt{5}}}\pi\right)}\approx 14\,427.7,\\
    d_2\left(\mathcal{O},\mathcal{P}_k+D_{\mathcal{O}}\right)&=\frac{\sqrt{3979\pi}}{3},\\
    6\pi d_2^2\left(\mathcal{O},\mathcal{P}_k+D_{\mathcal{O}}\right)&=\frac{7958\pi^2}{3}\approx 26\, 180.8.
\end{align*}
This confirms our inequality for the chosen $h$, and the value of the deficit
$\Delta_5(\mathcal O)$ lies above the two quantitative
lower bounds obtained from our stability estimates.
\end{ex}




\bibliographystyle{amsalpha}

\end{document}